\documentclass[12pt]{amsart}
\usepackage{amsmath,amssymb,hyperref,geometry,graphicx,ytableau,fp,mathdots,amsthm,tikz,tikz-3dplot}
\usetikzlibrary{cd}
\usepackage{pgfplots}
\usepackage{float}
\usepgfplotslibrary{fillbetween}
\usepackage[reftex]{theoremref}
\usetikzlibrary{arrows}
\usepackage{subfiles}

\newtheorem{theorem}[equation]{Theorem}
\newtheorem{lemma}[equation]{Lemma}
\newtheorem{corollary}[equation]{Corollary}
\newtheorem{proposition}[equation]{Proposition}
\theoremstyle{definition}

\newtheorem{remark}[equation]{Remark}

\numberwithin{equation}{section}

\newcommand{\R}{\mathbb{R}}
\newcommand{\Z}{\mathbb{Z}}

\newcommand{\I}{\mathcal{I}}
\newcommand{\J}{\mathcal{J}}
\newcommand{\cL}{\mathcal{L}}
\newcommand{\cF}{\mathcal{F}}
\newcommand{\sM}{\mathsf{M}}

\newcommand{\Cub}{\operatorname{Cub}}
\newcommand{\oCub}{\overline{\Cub}}
\newcommand{\Vol}{\operatorname{Vol}}

\newcommand{\MVol}{\operatorname{MVol}}
\newcommand{\pr}{\mathrm{pr}}
\newcommand{\cl}{\mathrm{cl}}
\newcommand{\rk}{\mathrm{rk}}
\newcommand{\Inn}{\mathrm{Inn}}

\newcommand{\conv}{\mathrm{conv}}
\newcommand{\F}{\mathrm{F}}
\newcommand{\N}{\mathrm{N}} 
\newcommand{\Hess}{\mathrm{Hess}}

\newcommand*{\defeq }{\mathrel{\vcenter{\baselineskip0.5ex \lineskiplimit0pt
                     \hbox{\scriptsize.}\hbox{\scriptsize.}}}%
                     =}

\begin{document}
\title{Mixed volumes of normal complexes}
\author[L.~Nowak]{Lauren Nowak}
\address{Department of Mathematics, University of Washington}
\email{lnowak@uw.edu}
\author[P.~O'Melveny]{Patrick O'Melveny}
\address{Department of Mathematics, San Francisco State University}
\email{pomelveny@mail.sfsu.edu}
\author[D.~Ross]{Dustin Ross}
\address{Department of Mathematics, San Francisco State University}
\email{rossd@sfsu.edu}

\begin{abstract}
Normal complexes are orthogonal truncations of polyhedral fans. In this paper, we develop the study of mixed volumes for normal complexes. Our main result is a sufficiency condition that ensures when the mixed volumes of normal complexes associated to a given fan satisfy the Alexandrov--Fenchel inequalities. By specializing to Bergman fans of matroids, we give a new proof of the Heron--Rota--Welsh Conjecture as a consequence of the Alexandrov--Fenchel inequalities for normal complexes.
\end{abstract}

\maketitle
\vspace{-20bp}

\section{Introduction}

The Alexandrov--Fenchel inequalities lie at the heart of convex geometry, asserting that, for any convex bodies $P_{\heartsuit},P_{\diamondsuit},P_3\dots,P_d\in\R^d$, their mixed volumes satisfy
\[
\MVol(P_{\heartsuit},P_{\diamondsuit},P_3,\dots,P_d)^2\geq \MVol(P_{\heartsuit},P_{\heartsuit},P_3,\dots,P_d)\MVol(P_{\diamondsuit},P_{\diamondsuit},P_3,\dots,P_d).
\]
This paper is centered around developing an analogue of the Alexandrov--Fenchel inequalities in a decidedly nonconvex setting. The geometric objects of interest to us are normal complexes, which were recently introduced by A. Nathanson and the third author \cite{NR}. Given a pure simplicial fan $\Sigma$, a normal complex associated to $\Sigma$ is, roughly speaking, a polyhedral complex obtained by truncating each cone of $\Sigma$ with half-spaces perpendicular to the rays of $\Sigma$. The choice of where to place the truncating half-spaces results in a family of normal complexes associated to each fan $\Sigma$, and the question that motivates this work is: \emph{for a given fan $\Sigma$, do the mixed volumes of the associated normal complexes satisfy the Alexandrov--Fenchel inequalities?} Our main result (Theorem~\ref{thm:reduce}) describes two readily verifiable conditions on $\Sigma$ that guarantee an affirmative answer to this question.

One of the motivations for studying mixed volumes of normal complexes is that, in the special setting of tropical fans, they correspond to mixed degrees of divisors in associated Chow rings. Thus, Alexandrov--Fenchel inequalities for normal complexes lead to nontrivial numerical inequalities in these Chow rings. A class of tropical fans that have garnered a great deal of attention in recent years are Bergman fans of matroids, and one application of our main result (Theorem~\ref{thm:bergman}) is that normal complexes associated to Bergman fans of matroids satisfy the Alexandrov--Fenchel inequalities. Translating these inequalities back to matroid Chow rings, we obtain a volume-theoretic proof of the log-concavity of characteristic polynomials of matroids, a result that was conjectured by Heron, Rota, and Welsh \cite{Rota, Heron, Welsh} and first proved by Adiprasito, Huh, and Katz \cite{AHK}.

\subsection{Overview of the paper}

We begin in Section 2 by briefly recalling the construction of normal complexes and their volumes. Normal complexes, denoted $C_{\Sigma,*}(z)$, depend on a marked simplicial $d$-fan $\Sigma$ in a vector space $N_\R$ with an inner product $*\in\Inn(N_\R)$, as well as a choice of pseudocubical truncating values $z\in\oCub(\Sigma,*)\subseteq\R^{\Sigma(1)}$. The volume of $C_{\Sigma,*}(z)$, denoted $\Vol_{\Sigma,\omega,*}(z)$, where $\omega$ is a weight function on the top-dimensional cones of $\Sigma$, is defined as the weighted sum of the volumes of the maximal polytopes in $C_{\Sigma,*}(z)$. We recall the main result of \cite{NR}, which asserts that, if $(\Sigma,\omega)$ is a tropical fan, then
\begin{equation}\label{intro:vol=deg}
\Vol_{\Sigma,\omega,*}(z)=\deg_{\Sigma,\omega}(D(z)^d)\;\;\;\text{ where }\;\;\;D(z)=\sum_{\rho\in\Sigma(1)}z_\rho X_\rho \in A^1(\Sigma).
\end{equation}

In Section 3, we introduce mixed volumes of normal complexes $C_{\Sigma,*}(z_1),\dots,C_{\Sigma,*}(z_d)$, denoted $\MVol_{\Sigma,\omega,*}(z_1,\dots,z_d)$, which are weighted sums of mixed volumes of maximal polytopes. Analogous to mixed volumes in convex geometry, we show that mixed volumes of normal complexes are characterized by being symmetric, multilinear, and normalized by volume (Proposition~\ref{prop:mvolchar}). Furthermore, we prove that mixed volumes are nonnegative on the pseudocubical cone $\oCub(\Sigma,*)$ and positive on the cubical cone $\Cub(\Sigma,*)$ (Proposition~\ref{prop:positive}). For all tropical fans $(\Sigma,\omega)$, 
we leverage \eqref{intro:vol=deg} to show (Theorem~\ref{thm:mvol=mdeg}) that
\begin{equation}\label{intro:mvol=mdeg}
\MVol_{\Sigma,\omega,*}(z_1,\dots,z_d)=\deg_{\Sigma,\omega}(D(z_1)\cdots D(z_d)).
\end{equation}

In Section 4, we develop the face structure of normal complexes, closely paralleling the classical face structure of polytopes. In particular, the faces of a normal complex $C_{\Sigma,*}(z)$ are indexed by cones $\tau\in\Sigma$, and each face is obtained as the intersection of $C_{\Sigma,*}(z)$ with the truncating hyperplanes indexed by the rays of $\tau$. We describe how each face can, itself, be viewed as a normal complex associated to the star fan $\Sigma^\tau$, and use this to define (mixed) volumes of faces. Our main result of this section (Proposition~\ref{prop:pyramidmixed}), shows how mixed volumes of normal complexes can be computed in terms of mixed volumes of facets.

In Section 5, we introduce what it means for a triple $(\Sigma,\omega,*)$ to be AF---namely, that the mixed volumes of cubical values satisfy the Alexandrov--Fenchel inequalities. Our main result (Theorem~\ref{thm:reduce}), inspired by work of Cordero-Erausquin, Klartag, Merigot, and Santambrogio \cite{OneMoreProof} and Br\"and\'en and Leake \cite{BrandenLeake}, states that $(\Sigma,\omega,*)$ is AF if (i) all star fans $\Sigma^\tau$ of dimension at least three remain connected after removing the origin and (ii) the quadratic volume polynomials associated to the two-dimensional star fans of $\Sigma$ have exactly one positive eigenvalue. In fact, under these conditions, we argue that the volume polynomial $\Vol_{\Sigma,\omega,*}(z)$ is $\Cub(\Sigma,*)$-Lorentzian, which then implies that $(\Sigma,\omega,*)$ is AF.

In Section 6, we briefly recall relevant notions regarding matroids and Bergman fans, and then we use Theorem~\ref{thm:reduce} to prove that Bergman fans of matroids are AF (Theorem~\ref{thm:bergman}). We conclude the paper by deducing the Heron--Rota--Welsh Conjecture as a consequence of the Alexandrov--Fenchel inequalities for normal complexes.

\subsection{Relation to other work}

Since the original proof of the Heron--Rota--Welsh Conjecture by Adiprasito, Huh, and Katz \cite{AHK}, there have been a number of alternative proofs, generalizations, and exciting related developments (an incomplete list includes \cite{BHMPW,BHMPW2,BES,ADH,AP1,AP2,BrandenHuh,AGVI,ALGVII,ALGVIII,CP}). We view the volume-theoretic approach in this paper as a new angle from which to view log-concavity of characteristic polynomials of matroids, but we also want to acknowledge that our methods share features of and are indebted to the approaches of several other teams of mathematicians. In particular, our methods rely on the Chow-theoretic interpretation of characteristic polynomials of matroids, proved by Huh and Katz \cite{HuhKatz}, which was central in the original proof of Adiprasito, Huh, and Katz \cite{AHK}, as well as in the subsequent proofs by Braden, Huh, Matherne, Proudfoot, and Wang \cite{BHMPW} and Backman, Eur, and Simpson \cite{BES}. In addition, our methods prove that volume polynomials are Lorentzian, which is also a central feature in the methods of both Backman, Eur, and Simpson \cite{BES} and Br\"and\'en and Leake \cite{BrandenLeake}. We note that, while the methods of \cite{BES} and \cite{BrandenLeake} seem to be tailored primarily for matroids, our methods readily extend to the more general setting of tropical intersection theory (this extension will be spelled out in a forthcoming work of the third author). By adding a new volume-theoretic approach to the Heron--Rota--Welsh Conjecture to the literature, we hope that this paper will serve to welcome a new batch of geometrically-minded folks into the fold of this flourishing area of research, opening the door for further developments.

\subsection{Acknowledgements}

The authors would like to express their gratitude to Federico Ardila, Matthias Beck, Emily Clader, Chris Eur, and Serkan Ho\c{s}ten for sharing insights related to this project. This work was supported by a grant from the National Science Foundation: DMS-2001439.

\section{Background on normal complexes}\label{sec:background}

In this section, we establish notation, conventions, and preliminary results regarding polyhedral fans and normal complexes.

\subsection{Fan definitions and conventions}

Let $N_\R$ be a real vector spaces of dimension $n$. Given a polyhedral fan $\Sigma\subseteq N_\R$, we denote the $k$-dimensional cones of $\Sigma$ by $\Sigma{(k)}$. Let $\preceq$ denote the face containment relation among the cones of $\Sigma$, and for each cone $\sigma\in\Sigma$, let $\sigma(k)\subseteq\Sigma(k)$ denote the $k$-dimensional faces of $\sigma$. For any cone $\sigma$, let $\sigma^\circ$ denote the relative interior of $\sigma$ and denote the linear span of $\sigma$ by $N_{\sigma,\R}\subseteq N_\R$.

We say that a fan $\Sigma$ is \textbf{pure} if all of the maximal cones in $\Sigma$ have the same dimension. We say that $\Sigma$ is \textbf{marked} if we have chosen a distinguished generating vector $0\neq u_\rho\in \rho$ for each ray $\rho\in\Sigma(1)$. Henceforth, we assume that all fans are pure, polyhedral, and marked, and we use the term \textbf{$d$-fan} to refer to a pure, polyhedral, marked fan of dimension $d$.

We say that $\Sigma$ is \textbf{simplicial} if $\dim(N_{\sigma,\R})=|\sigma(1)|$ for all $\sigma\in\Sigma$. The faces of a simplicial cone $\sigma$ are in bijective correspondence with the subsets of $\sigma(1)$. For every face containment $\tau\preceq\sigma$ in a simplicial fan $\Sigma$, let $\sigma\setminus\tau$ denote the face of $\sigma$ with rays $\sigma(1)\setminus\tau(1)$. Given two faces $\tau,\pi\preceq\sigma$, denote by $\tau\cup\pi$ the face of $\sigma$ with rays $\tau(1)\cup\pi(1)$.

Given a simplical $d$-fan $\Sigma$ and a weight function $\omega:\Sigma(d)\rightarrow\R_{>0}$, we say that the pair $(\Sigma,\omega)$ is a \textbf{tropical fan} if it satisfies the weighted balancing condition:
\[
\sum_{\sigma\in\Sigma(d)\atop \tau\prec\sigma}\omega(\sigma)u_{\sigma\setminus\tau}\in N_{\tau,\R} \;\;\; \text{ for all } \;\;\; \tau\in\Sigma(d-1).
\]
While the definition of tropical fans can be generalized to nonsimplicial fans, we will assume throughout this paper that all tropical fans are simplicial. If $\omega(\sigma)=1$ for all $\sigma\in\Sigma(d)$, we say that $\Sigma$ is \textbf{balanced} and we omit $\omega$ from the notation.

\subsection{Chow rings and degree maps}
Let $M_\R$ denote the dual of $N_\R$ and let $\langle-,-\rangle$ be the duality pairing. Given a simplicial fan $\Sigma\subseteq N_\R$, the \textbf{Chow ring of $\Sigma$} is defined by
\[
A^\bullet(\Sigma)\defeq \frac{\R\big[x_\rho\;|\;\rho\in\Sigma{(1)}\big]}{\I+\J}
\]
where
\[
\I\defeq \big\langle x_{\rho_1}\cdots x_{\rho_k}\;|\;\R_{\geq 0}\{\rho_1,\dots,\rho_k\}\notin\Sigma\big\rangle \;\;\; \text{ and } \;\;\; \J\defeq \bigg\langle \sum_{\rho\in\Sigma{(1)}}\langle v,u_\rho\rangle x_\rho\;\bigg|\;v\in M_\R\bigg\rangle.
\]
As both $\I$ and $\J$ are homogeneous, the Chow ring $A^\bullet(\Sigma)$ is a graded ring, and we denote by $A^k(\Sigma)$ the subgroup of homogeneous elements of degree $k$. We denote the generators of $A^\bullet(\Sigma)$ by $X_\rho\defeq [x_\rho]\in A^1(\Sigma)$, and for any $\sigma\in\Sigma(k)$, we define
\[
 X_\sigma\defeq \prod_{\rho\in\sigma(1)}X_\rho\in A^k(\Sigma).
\]
If $\Sigma$ is a simplicial $d$-fan, then every element of $A^k(\Sigma)$ can be written as a linear combination of $X_\sigma$ with $\sigma\in\Sigma(k)$ (see, for example, \cite[Proposition~5.5]{AHK}). It follows that $A^k(\Sigma)=0$ for all $k>d$. If, in addition, $(\Sigma,\omega)$ is tropical, then there is a well-defined \textbf{degree map}
\[
\deg_{\Sigma,\omega}:A^d(\Sigma)\rightarrow\R
\]
such that $\deg_{\Sigma,\omega}(X_\sigma)=\omega(\sigma)$ for every $\sigma\in\Sigma(d)$ (see, for example, \cite[Proposition~5.6]{AHK}).

\subsection{Normal complexes}

We now recall the construction of normal complexes from \cite{NR}. In addition to a simplicial $d$-fan $\Sigma\subseteq N_\R$, the normal complex construction requires an additional choice of an inner product $*\in\Inn(N_\R)$ and a value $z\in\R^{\Sigma(1)}$. Given such a $*$ and $z$, we define a set of hyperplanes and half-spaces in $N_\R$ associated to each $\rho\in\Sigma$ by
\[
H_{\rho,*}(z)\defeq \{v\in N_\R \mid v*u_\rho= z_\rho\}\;\;\;\text{ and }\;\;\;H_{\rho,*}^-(z)\defeq \{v\in N_\R \mid v*u_\rho\leq z_\rho\}.
\]
We then define polytopes $P_{\sigma,*}(z)$, one for each $\sigma\in\Sigma$, by
\[
P_{\sigma,*}(z)\defeq \sigma\cap\bigcap_{\rho\in\sigma(1)}H_{\rho,*}^-(z).
\]
Notice that $P_{\sigma,*}(z)$ is simply a truncation of the cone $\sigma$ by hyperplanes that are normal to the rays of $\sigma$---what it means to be normal is determined by $*$, and the locations of the normal hyperplanes along the rays of the cone are determined by $z$. We would like to construct a polytopal complex from these polytopes, but in general, they do not meet along faces. To ensure that they meet along faces, we require a compatibility between $z$ and $*$.

For each $\sigma\in\Sigma$, let $w_{\sigma,*}(z)\in N_{\sigma,\R}$ be the unique vector such that $w_{\sigma,*}(z)*u_\rho=z_\rho$ for all $\rho\in\sigma(1)$. That such a vector exists and is unique follows from the fact that the vectors $u_\rho$ with $\rho\in\sigma(1)$ are linearly independent---this is equivalent to the simplicial hypothesis. We then say that $z$ is \textbf{cubical (pseudocubical) with respect to $(\Sigma,*)$} if
\[
w_{\sigma,*}(z)\in\sigma^\circ\;\;\;(w_{\sigma,*}(z)\in\sigma)\;\;\;\text{ for all }\;\;\;\sigma\in\Sigma.
\]
In other words, the pseudocubical values are those values of $z$ for which the truncating hyperplanes intersect within each cone, and the cubical values are those for which they intersect in the relative interior of each cone. The collection of cubical values are denoted $\Cub(\Sigma,*)\subseteq\R^{\Sigma(1)}$ and the pseudocubical values are denoted $\oCub(\Sigma,*)\subseteq\R^{\Sigma(1)}$.

We now summarize key results from \cite{NR} that will be necessary for the developments in this paper (see \cite[Propositions~3.2, 3.3, and 3.7 ]{NR}).

\begin{proposition}\label{prop:normalcomplexprelims}
Let $\Sigma\subseteq N_\R$ be a simplicial $d$-fan and let $*\in\Inn(N_\R)$ be an inner product.
\begin{enumerate}
\item The set $\oCub(\Sigma,*)\subseteq\R^{\Sigma(1)}$ is a polyhedral cone with $\oCub(\Sigma,*)^\circ=\Cub(\Sigma,*)$.
\item For $z\in\oCub(\Sigma,*)$, the vertices of $P_{\sigma,*}(z)$ are $\{w_{\tau,*}(z)\mid\tau\preceq\sigma\}$.
\item For $z\in\oCub(\Sigma,*)$, the polytopes $P_{\sigma,*}(z)$ meet along faces.
\end{enumerate}
\end{proposition}

For any polytope $P$, let $\widehat P$ denote the set of all faces of $P$. The third part of Proposition~\ref{prop:normalcomplexprelims} implies that
\[
C_{\Sigma,*}(z)\defeq \bigcup_{\sigma\in\Sigma(d)}\widehat{P_{\sigma,*}(z)}
\]
is a polytopal complex whenever $z\in\oCub(\Sigma,*)$, and this polytopal complex is called the \textbf{normal complex of $\Sigma$ with respect to $*$ and $z$.}

Below, we depict a two-dimensional tropical fan and an associated normal complex. The fan is comprised of nine two-dimensional cones glued along faces, and each of these nine cones corresponds to a quadrilateral in the normal complex.

\begin{center}
\tdplotsetmaincoords{68}{55}
\begin{tikzpicture}[scale=2,tdplot_main_coords]
\draw[draw=blue!20,fill=blue!20,fill opacity=0.5]  (0,0, 0)-- (1, 0, 0) -- (1, 1, 1) -- cycle;
\draw[draw=blue!20,fill=blue!20,fill opacity=0.5]  (0,0, 0)-- (0, 1, 0) -- (1, 1, 1) -- cycle;
\draw[draw=blue!20,fill=blue!20,fill opacity=0.5]  (0,0, 0)-- (0, 0, 1) -- (1, 1, 1) -- cycle;
\draw[draw=blue!20,fill=blue!20,fill opacity=0.5] (0,0, 0)-- (1, 0, 0) -- (0, -1, -1) -- cycle;
\draw[draw=blue!20,fill=blue!20,fill opacity=0.5] (0,0, 0)-- (0, 1, 0) -- (-1, 0, -1) -- cycle;
\draw[draw=blue!20,fill=blue!20,fill opacity=0.5] (0,0, 0)-- (0, 0, 1) -- (-1, -1, 0) -- cycle;
\draw[draw=blue!20,fill=blue!20,fill opacity=0.5] (0,0, 0)-- (-1, -1, 0) -- (-1, -1, -1) -- cycle;
\draw[draw=blue!20,fill=blue!20,fill opacity=0.5] (0,0, 0)-- (-1, 0, -1) -- (-1, -1, -1) -- cycle;
\draw[draw=blue!20,fill=blue!20,fill opacity=0.5] (0,0, 0)-- (0, -1, -1) -- (-1, -1, -1) -- cycle;
\draw[->] (0,0,0) -- (1,0,0);
\draw[->,gray] (0,0,0) -- (0,1,0); 
\draw[->] (0,0,0) -- (0,0,1);
\draw[->] (0,0,0) -- (1,1,1);
\draw[->] (0,0,0) -- (-1,-1,-1);
\draw[->] (0,0,0) -- (-1,-1,0);
\draw[->,gray] (0,0,0) -- (-1,0,-1); 
\draw[->] (0,0,0) -- (0,-1,-1); 
\end{tikzpicture}
\hspace{70bp}
\begin{tikzpicture}[scale=1.1,tdplot_main_coords]
\draw[draw=green!20, fill=green!20, fill opacity=.5] (0,0,0) -- (0, 0, 1.6) -- (1, 1, 1.6) -- (1.2, 1.2, 1.2) -- cycle; 
\draw[draw=green!20, fill=green!20, fill opacity=.5] (0,0,0) -- (0, 1.6, 0) -- (1, 1.6, 1) -- (1.2, 1.2, 1.2) -- cycle; 
\draw[draw=green!20, fill=green!20, fill opacity=.5] (0,0,0) -- (1.6, 0, 0) -- (1.6, 1, 1) -- (1.2, 1.2, 1.2) -- cycle; 
\draw[draw=green!20, fill=green!20, fill opacity=.5] (0,0,0) -- (0, 0, 1.6) -- (-1.6, -1.6, 1.6) -- (-1.6, -1.6, 0) -- cycle; 
\draw[draw=green!20, fill=green!20, fill opacity=.5] (0,0,0) -- (0, 1.6, 0) -- (-1.6, 1.6, -1.6) -- (-1.6, 0, -1.6) -- cycle; 
\draw[draw=green!20, fill=green!20, fill opacity=.5] (0,0,0) -- (1.6, 0, 0) -- (1.6, -1.6, -1.6) -- (0, -1.6, -1.6) -- cycle; 
\draw[draw=green!20, fill=green!20, fill opacity=.5] (0,0,0) -- (-1.6, -1.6, 0) -- (-1.6, -1.6, -.4) -- (-1.2, -1.2, -1.2) -- cycle; 
\draw[draw=green!20, fill=green!20, fill opacity=.5] (0,0,0) -- (-1.6, 0, -1.6) -- (-1.6, -.4, -1.6) -- (-1.2, -1.2, -1.2) -- cycle; 
\draw[draw=green!20, fill=green!20, fill opacity=.5] (0,0,0) -- (0, -1.6, -1.6) -- (-.4, -1.6, -1.6) -- (-1.2, -1.2, -1.2) -- cycle; 

\draw[thick] (0, 0, 1.6) -- (1, 1, 1.6) -- (1.2, 1.2, 1.2);
\draw[dashed] (0, 1.6, 0) -- (1, 1.6, 1) -- (1.2, 1.2, 1.2);
\draw[thick] (.7, 1.6, .7) -- (1, 1.6, 1) -- (1.2, 1.2, 1.2);
\draw[thick] (1.6, 0, 0) -- (1.6, 1, 1) -- (1.2, 1.2, 1.2);
\draw[thick] (0, 0, 1.6)  -- (-1.6,  -1.6, 1.6)  -- (-1.6, -1.6, 0);
\draw[dashed] (0, 1.6, 0) -- (-1.6,  1.6, -1.6) -- (-1.6, 0, -1.6);
\draw[thick] (1.6, 0, 0) -- (1.6, -1.6, -1.6) -- (0, -1.6, -1.6);
\draw[thick] (-1.6, -1.6, 0) -- (-1.6, -1.6, -.4) -- (-1.2, -1.2, -1.2);
\draw[dashed] (-1.6, 0, -1.6) -- (-1.6, -.4, -1.6) -- (-1.2, -1.2, -1.2);
\draw[thick] (0, -1.6, -1.6) -- (-.4, -1.6, -1.6) -- (-1.2, -1.2, -1.2);

\draw[thick] (0, 0, 0) -- (1.2, 1.2, 1.2);
\draw[thick] (0, 0, 0) -- (0, 0, 1.6);
\draw[dashed] (0, 0, 0) -- (0, 1.6, 0);
\draw[thick] (0, 0, 0) -- (1.6, 0, 0);
\draw[thick] (0, 0, 0) -- (0, -1.6, -1.6);
\draw[dashed] (0, 0, 0) -- (-1.6, 0, -1.6);
\draw[thick] (0, 0, 0) -- (-1.6, -1.6, 0);
\draw[thick] (0, 0, 0) -- (-1.2, -1.2, -1.2);
\end{tikzpicture}
\end{center}
The next pair of images depict a three-dimensional fan comprised of two maximal cones meeting along a two-dimensional face, and a corresponding normal complex. While this fan is not tropical, the reader is welcome to view this image as just one small piece of a three-dimensional tropical fan in some higher-dimensional vector space.

\begin{center}
\begin{tikzpicture}
\node[] at (0,0) {\includegraphics[scale=.3]{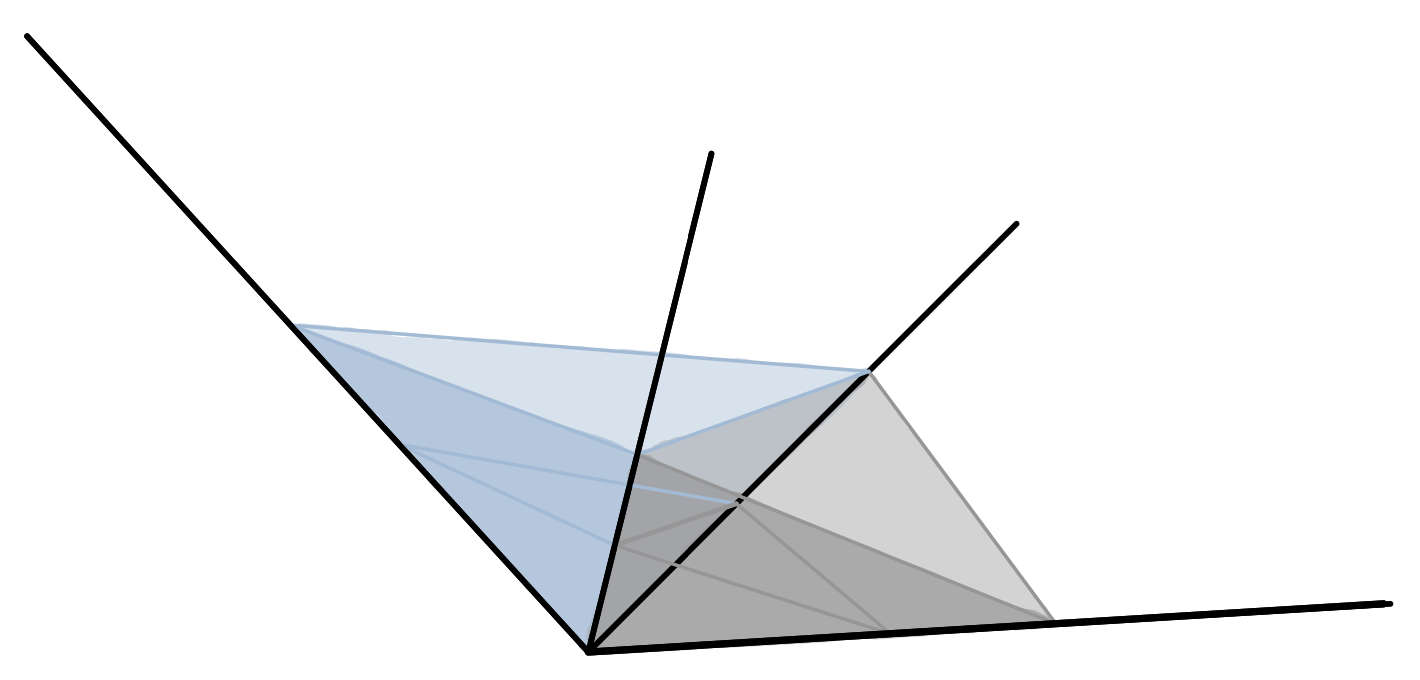}};
\node[] at (8,0) {\includegraphics[scale=.3]{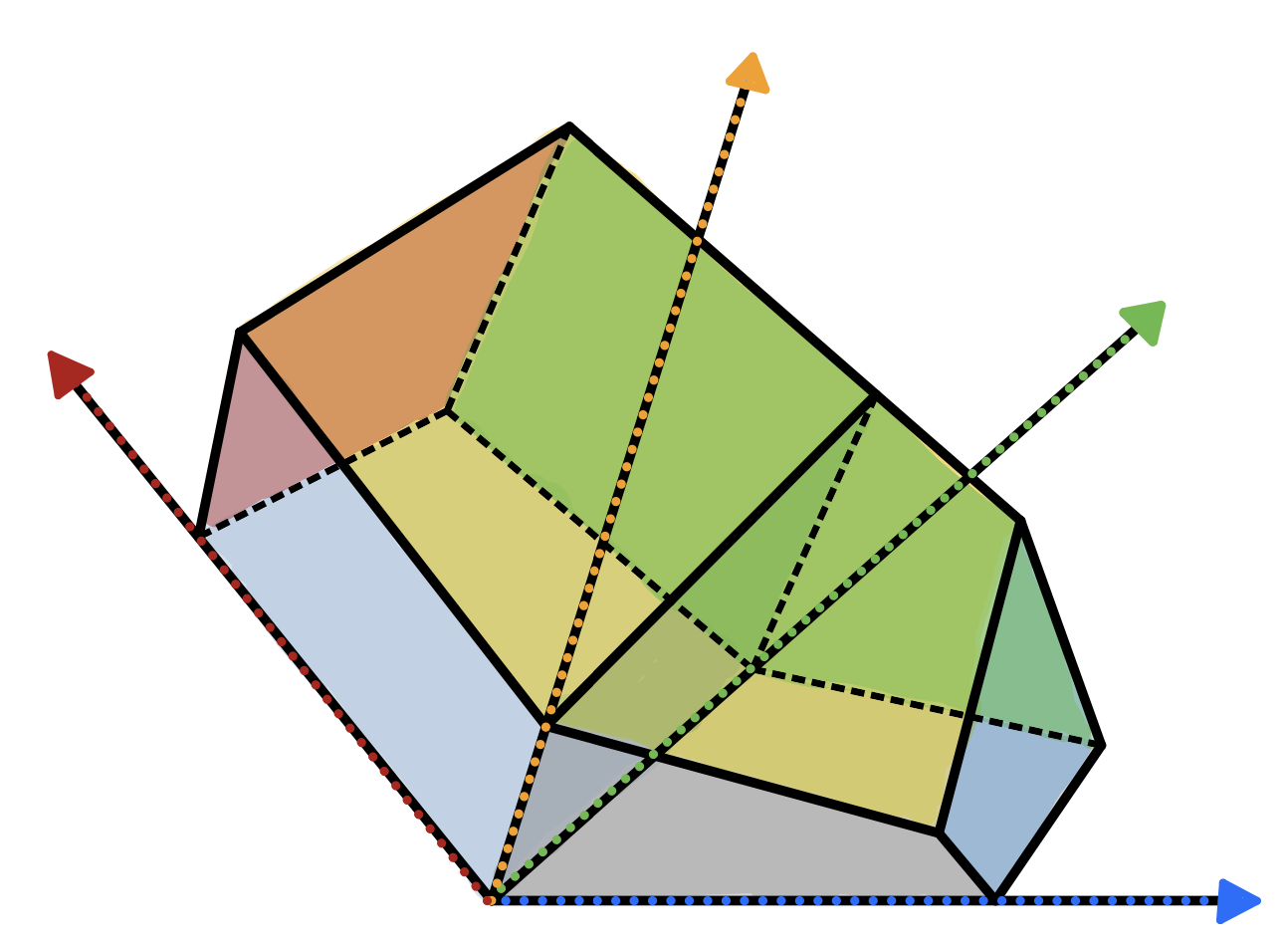}};
\end{tikzpicture}
\end{center}

\subsection{Volumes of normal complexes}

Let $\Sigma\subseteq N_\R$ be a simplicial $d$-fan, $*\in\Inn(N_\R)$ an inner product, and $z\in\oCub(\Sigma,*)$ a pseudocubical value. Informally, the volume of the normal complex $C_{\Sigma,*}(z)$ is the sum of the volumes of the polytopes $P_{\sigma,*}(z)$ with $\sigma\in\Sigma(d)$; however, some care is required in specifying what we mean by \emph{volume} in each subspace $N_{\sigma,\R}$.

For each cone $\sigma\in\Sigma$, define the discrete subgroup
\[
N_\sigma\defeq \mathrm{span}_\Z(u_\rho\;|\;\rho\in\sigma(1))\subseteq N_\R,
\]
and let $M_\sigma$ denote its dual: $M_\sigma\defeq \mathrm{Hom}_\Z(N_\sigma,\Z)\subseteq M_{\sigma,\R}\defeq \mathrm{Hom}_\R(N_{\sigma,\R},\R)$. Using the inner product $*$, we can identify $M_{\sigma,\R}$ with $N_{\sigma,\R}$ and thus, we can view $M_\sigma$ as a lattice in $N_{\sigma,\R}$. For each $\sigma\in\Sigma$, let
\[
\Vol_\sigma:\big\{\text{polytopes in }N_{\sigma,\R}\big\}\rightarrow\R_{\geq 0}
\]
be the volume function determined by the property that a fundamental simplex of the lattice $M_\sigma\subseteq N_{\sigma,\R}$ has unit volume. Define the \textbf{volume of the normal complex $C_{\Sigma,*}(z)$}, denoted $\Vol_{\Sigma,*}(z)$ for brevity, as the sum of the volumes of the constituent $d$-dimensional polytopes:
\[
\Vol_{\Sigma,*}(z)\defeq \sum_{\sigma\in\Sigma(d)}\Vol_\sigma(P_{\sigma,*}(z)).
\]
In slightly more generality, suppose that $\omega:\Sigma(d)\rightarrow \R_{>0}$ is a weight function on the maximal cones of $\Sigma$. The \textbf{volume of the normal complex $C_{\Sigma,*}(z)$ weighted by $\omega$} is defined by
\[
\Vol_{\Sigma,\omega,*}(z)\defeq \sum_{\sigma\in\Sigma(d)}\omega(\sigma)\Vol_\sigma(P_{\sigma,*}(z)).
\]

The main result of \cite{NR} is a Chow-theoretic interpretation of the weighted volumes of normal complexes, valid whenever $(\Sigma,\omega)$ is tropical.

\begin{theorem}[{\cite[Theorem~6.3]{NR}}]\label{thm:vol=deg}
Let $(\Sigma,\omega)$ be a tropical $d$-fan, $*\in\Inn(N_\R)$ an inner product, and $z\in\oCub(\Sigma,*)$ a pseudocubical value. Then
\[
\Vol_{\Sigma,\omega,*}(z)=\deg_{\Sigma,\omega}(D(z)^d)
\]
where
\[
D(z)=\sum_{\rho\in\Sigma(1)} z_\rho X_\rho\in A^1(\Sigma).
\]
\end{theorem}

\section{Mixed Volumes of Normal Complexes}

Our first aim in this paper is to enhance Theorem~\ref{thm:vol=deg} to a statement about mixed volumes. In order to do this, we briefly recall the classical theory of mixed volumes, for which we recommend the comprehensive text by Schneider \cite{Schneider} as a reference.

\subsection{Mixed volumes of polytopes}

Mixed volumes are the natural result of combining the notion of volume with the operation of Minkowski addition. We start with a $d$-dimensional real vector space $V$ and a volume function $\Vol:\{\text{polytopes in }V\}\rightarrow\R_{\geq 0}$.  The \textbf{mixed volume function} 
\[
\MVol:\{\text{polytopes in } V\}^d\rightarrow\R_{\geq 0}
\]
is the unique function determined by the following three properties.
\begin{itemize}
\item (Symmetry) For any permutation $\pi\in S_d$,
\[
\MVol(P_1,\dots,P_d)=\MVol(\pi(P_1,\dots,P_d)).
\] 
\item (Multilinearity) For any $i=1,\dots,d$ and $\lambda \in\R_{\geq 0}$, 
\begin{align*}
\MVol(P_1,\dots,\lambda P_i+P_i',\dots,P_d)=\lambda&\MVol(P_1,\dots,P_i,\dots,P_d)\\&+\MVol(P_1,\dots,P_i',\dots,P_d),
\end{align*}
where the linear combination of polytopes is defined by
\[
\lambda P_i+P_i'=\{\lambda v+w\mid v\in P_i,w\in P_i'\}.
\]
\item (Normalization) For any polytope $P$,
\[
\MVol(P,\dots,P)=\Vol(P).
\]
\end{itemize}
That such a mixed volume function exists and is unique is due to Minkowski \cite{Minkowski}, who proved that such a function exists more generally for convex bodies, not just for polytopes.

\subsection{Mixed volumes of normal complexes}

We now define a notion of mixed volumes of normal complexes. Let $\Sigma\subseteq N_\R$ be a simplicial $d$-fan and let $*\in\Inn(N_\R)$ be an inner product. Given pseudocubical values $z_1,\dots,z_d\in\oCub(\Sigma,*)$, we define the \textbf{mixed volume of the normal complexes $C_{\Sigma,*}(z_1),\dots,C_{\Sigma,*}(z_d)$}, denoted $\MVol_{\Sigma,*}(z_1,\dots,z_d)$ for brevity, by
\[
\MVol_{\Sigma,*}(z_1,\dots,z_d)\defeq \sum_{\sigma\in\Sigma(d)}\MVol_\sigma(P_{\sigma,*}(z_1),\dots,P_{\sigma,*}(z_d)).
\]
In other words, the mixed volume is the sum of the mixed volumes of the polytopes associated to the top-dimensional cones of $\Sigma$. More generally, if $\omega:\Sigma(d)\rightarrow \R_{>0}$ is a weight function, then the \textbf{mixed volume of the normal complexes $C_{\Sigma,*}(z_1),\dots,C_{\Sigma,*}(z_d)$ weighted by $\omega$} is defined by
\[
\MVol_{\Sigma,\omega,*}(z_1,\dots,z_d)\defeq \sum_{\sigma\in\Sigma(d)}\omega(\sigma)\MVol_\sigma(P_{\sigma,*}(z_1),\dots,P_{\sigma,*}(z_d)).
\]
In order to verify that this is a meaningful notion of mixed volumes for normal complexes, we check that it is characterized by an analogue of the three characterizing properties of mixed volumes of polytopes.

\begin{proposition}\label{prop:mvolchar}
Let $\Sigma\subseteq N_\R$ be a simplicial $d$-fan, $*\in\Inn(N_\R)$ an inner product, and $\omega:\Sigma(d)\rightarrow\R_{>0}$ a weight function.
\begin{enumerate}
\item For any $z_1,\dots,z_d\in\oCub(\Sigma,*)$ and $\pi\in S_d$, 
\[
\MVol_{\Sigma,\omega,*}(z_1,\dots,z_d)=\MVol_{\Sigma,\omega,*}(\pi(z_1,\dots,z_d)).
\]
\item For any $i=1,\dots,d$, and for any $z_1,\dots,z_i,z_i',\dots,z_d\in\oCub(\Sigma,*)$ and $\lambda\in\R_{\geq 0}$, 
\begin{align*}
\MVol_{\Sigma,\omega,*}(z_1,\dots,\lambda z_i+z_i',\dots,z_d)=\lambda &\MVol_{\Sigma,\omega,*}(z_1,\dots,z_i,\dots,z_d)\\
&+\MVol_{\Sigma,\omega,*}(z_1,\dots,z_i',\dots,z_d).
\end{align*}
\item For any $z\in\oCub(\Sigma,*)$, 
\[
\MVol_{\Sigma,\omega,*}(z,\dots,z)=\Vol_{\Sigma,\omega,*}(z).
\]
\end{enumerate}
Moreover, any function $\oCub(\Sigma,*)^d\rightarrow\R_{\geq 0}$ satisfying Properties \emph{(1) -- (3)} must be $\MVol_{\Sigma,\omega,*}$.
\end{proposition}

\begin{proof}
Given that
\[
\MVol_{\Sigma,\omega,*}(z_1,\dots,z_d)=\sum_{\sigma\in\Sigma(d)}\omega(\sigma)\MVol_\sigma(P_{\sigma,*}(z_1),\dots,P_{\sigma,*}(z_d))
\]
and the summands in the right-hand side are simply mixed volumes of polytopes, Properties~(1) and (3) follow from the symmetry and normalization properties of mixed volumes in the polytope setting. Moreover, once we prove that 
\begin{equation}\label{eq:polytopelinearity}
P_{\sigma,*}(\lambda z+z')=\lambda P_{\sigma,*}(z)+P_{\sigma,*}(z')
\end{equation}
for all $z,z'\in\oCub(\Sigma,*)$ and $\lambda\in\R_{\geq 0}$, then Property (2) also follows from the multilinearity property of mixed volumes in the polytope setting. Thus, it remains to prove \eqref{eq:polytopelinearity}, which we accomplish by proving both inclusions.

First, suppose that $v\in P_{\sigma,*}(\lambda z+z')$. By Proposition~\ref{prop:normalcomplexprelims}, the vertices of $P_{\sigma,*}(\lambda z+z')$ are $\{w_{\tau,*}(\lambda z+z')\mid\tau\preceq\sigma\}$, so we can write $v$ as a convex combination:
\begin{equation}\label{eq:vertexexpression}
v=\sum_{\tau\preceq\sigma}a_\tau\, w_{\tau,*}(\lambda z+z')\;\;\;\text{ for some }\;\;\;a_\tau\in\R_{\geq 0}\;\;\;\text{ with }\;\;\;\sum_{\tau\preceq\sigma}a_\tau=1.
\end{equation}
To prove that $v\in\lambda P_{\sigma,*}(z)+P_{\sigma,*}(z')$, our next step is to prove that the vertices are linear:
\begin{equation}\label{eq:linearws}
w_{\tau,*}(\lambda z+z')=\lambda w_{\tau,*}(z)+w_{\tau,*}(z').
\end{equation}
Since $w_{\tau,*}(\lambda z+z')$ is the unique vector in $N_{\tau,\R}$ with $w_{\tau,*}(\lambda z+z')*u_\rho=(\lambda z+z')_\rho$ for all $\rho\in\tau(1)$, proving \eqref{eq:linearws} amounts to proving that $\lambda w_{\tau,*}(z)+w_{\tau,*}(z')$ also satisfies these equations. Using bilinearity of the inner product and the definition of the $w$ vectors, we have
\begin{align*}
(\lambda w_{\tau,*}(z)+w_{\tau,*}(z'))*u_\rho&=\lambda w_{\tau,*}(z)*u_\rho+w_{\tau,*}(z')*u_\rho\\
&=\lambda z_\rho+z_\rho'\\
&=(\lambda z+z')_\rho.
\end{align*}
Therefore, \eqref{eq:linearws} holds, and substituting \eqref{eq:linearws} into \eqref{eq:vertexexpression} implies that
\[
v=\lambda\sum_{\tau\preceq\sigma}a_\tau w_{\tau,*}(z)+\sum_{\tau\preceq\sigma}a_\tau w_{\tau,*}(z')\in \lambda P_{\sigma,*}(z)+P_{\sigma,*}(z').
\]

To prove the other inclusion, suppose that $v\in \lambda P_{\sigma,*}(z)+P_{\sigma,*}(z')$. Then $v=\lambda w+w'$ for some $w\in P_{\sigma,*}(z)$ and $w'\in P_{\sigma,*}(z')$. This means that $w,w'\in\sigma$ and, in addition, $w\cdot u_\rho\leq z_{\rho}$ and $w'\cdot u_\rho\leq z_{\rho}'$ for all $\rho\in\sigma(1)$. Since $\sigma$ is a cone, $u=\lambda w+w'\in\sigma$ and, for every $\rho\in\sigma(1)$, we have
\begin{align*}
v*u_\rho&=(\lambda w+w')*u_\rho\\
&=\lambda w*u_\rho+w'*u_\rho\\
&\leq \lambda z_\rho+z_\rho',
\end{align*}
from which we conclude that $v\in P_{\sigma,*}(\lambda z+z')$.

Finally, to prove the final assertion of the proposition, suppose that $F:\oCub(\Sigma,*)^d\rightarrow\R_{\geq 0}$ satisfies Properties (1) -- (3). Our goal is to prove that $F(z_1,\dots,z_d)=\MVol_{\Sigma,\omega,*}(z_1,\dots,z_d)$ for any pseudocubical values $z_1,\dots,z_d\in\oCub(\Sigma,*)$. Set $z=\lambda_1 z_1+\dots+\lambda_d z_d$ with $\lambda_1,\dots,\lambda_d\in\R_{\geq 0}$ arbitrary. Property (3) implies that
\[
F(z,\dots,z)=\Vol_{\Sigma,\omega,*}(z)=\MVol_{\Sigma,\omega,*}(z,\dots,z).
\]
Using Properties (1) and (2) we can expand both the left- and right-hand sides of this equation as polynomials in $\lambda_1,\dots,\lambda_d$:
\begin{align*}
\sum_{k_1,\dots,k_d}{d\choose k_1,\dots,k_d}&F(\underbrace{z_1,\dots,z_1}_{k_1},\dots,\underbrace{z_d,\dots,z_d}_{k_d})\lambda_1^{k_1}\cdots\lambda_d^{k_d}\\
&=\sum_{k_1,\dots,k_d}{d\choose k_1,\dots,k_d}\MVol_{\Sigma,\omega,*}(\underbrace{z_1,\dots,z_1}_{k_1},\dots,\underbrace{z_d,\dots,z_d}_{k_d})\lambda_1^{k_1}\cdots\lambda_d^{k_d}
\end{align*}
Equating the coefficients of $\lambda_1\cdots\lambda_d$ in these two polynomials leads to the desired conclusion: $F(z_1,\dots,z_d)=\MVol_{\Sigma,\omega,*}(z_1,\dots,z_d)$.
\end{proof}

Our methods for studying Alexandrov--Fenchel inequalities will also require the following positivity result.

\begin{proposition}\label{prop:positive}
Let $\Sigma\subseteq N_\R$ be a simplicial $d$-fan, $*\in\Inn(N_\R)$ an inner product, and $\omega:\Sigma(d)\rightarrow\R_{>0}$ a weight function. Then 
\[
\MVol_{\Sigma,\omega,*}(z_1,\dots,z_d)\geq 0\;\;\;\text{ for all }\;\;\;z_1,\dots,z_d\in\oCub(\Sigma,*)
\]
and
\[
\MVol_{\Sigma,\omega,*}(z_1,\dots,z_d)> 0\;\;\;\text{ for all }\;\;\;z_1,\dots,z_d\in\Cub(\Sigma,*).
\]
\end{proposition}

\begin{proof}
The first statement follows from the definition of $\MVol_{\Sigma,\omega,*}$ and the nonnegativity of mixed volumes of polytopes \cite[Theorem~5.1.7]{Schneider}. For the second statement, we first observe that $z\in\Cub(\Sigma,*)$ implies that $P_{\sigma,*}(z)$ has dimension $d$ for every $\sigma\in\Sigma(d)$, which follows from the fact that   $P_{\sigma,*}(z)$ is combinatorially equivalent to a $d$-cube \cite[Proposition~3.8]{NR}. Thus, the second statement follows from the fact that mixed volumes of full-dimensional polytopes are strictly positive \cite[Theorem~5.1.8]{Schneider}.
\end{proof}

\subsection{Mixed volumes and mixed degrees}

We now extend Theorem~\ref{thm:vol=deg} to give a Chow-theoretic interpretation of mixed volumes of normal complexes associated to tropical fans.

\begin{theorem}\label{thm:mvol=mdeg}
Let $(\Sigma,\omega)$ be a tropical $d$-fan, let $*\in\Inn(N_\R)$ be an inner product, and let $z_1,\dots,z_d\in\oCub(\Sigma,*)$ be pseudocubical values. Then
\[
\MVol_{\Sigma,\omega,*}(z_1,\dots,z_d)=\deg_{\Sigma,\omega}(D(z_1)\cdots D(z_d)).
\]
\end{theorem}

\begin{proof}
By Proposition~\ref{prop:mvolchar}, it suffices to prove that the function
\begin{align*}
\oCub(\Sigma,*)^d&\rightarrow\R_{\geq 0}\\
(z_1,\dots,z_d)&\mapsto \deg_{\Sigma,\omega}(D(z_1)\cdots D(z_d))
\end{align*}
is symmetric, multilinear, and normalized by $\Vol_{\Sigma,\omega,*}$. Symmetry follows from the fact that $A^\bullet(\Sigma)$ is a commutative ring, multilinearity follows from the fact that $\deg_{\Sigma,\omega}:A^d(\Sigma)\rightarrow\R$ is a linear map, and normalization is the content of Theorem~\ref{thm:vol=deg}.
\end{proof}

\section{Faces of Normal Complexes}

In this section, we develop a face structure for normal complexes, analogous to the face structure of polytopes. Parallel to the polytope case, we will see that each face is obtained by intersecting the normal complex with supporting hyperplanes, that each face can, itself, be viewed as a normal complex, and that a face of a face is, itself, a face. We then prove fundamental properties relating (mixed) volumes of normal complexes to the (mixed) volumes of their facets, which perfectly parallel central results in the classical polytope setting.

\subsection{Orthogonal decompositions}

The face construction for normal complexes makes heavy use of an orthogonal decomposition of $N_\R$ associated to each cone $\tau\in\Sigma$, which we now describe. Associated to each $\tau\in\Sigma$, we have already met the subspace $N_{\tau,\R}\subseteq\N_\R$, which is the linear span of $\tau$, and we now introduce notation for the quotient space
\[
N_\R^\tau\defeq N_\R/N_{\tau,\R}.
\]
With the inner product $*$, we may identify $N^\tau_\R$ as the orthogonal complement of $N_{\tau,\R}$:
\[
N^\tau_\R=N_{\tau,\R}^\perp=\{v\in N_\R\mid v*u=0\text{ for all }u\in N_{\tau,\R}\}\subseteq N_\R,
\]
allowing us to decompose $N_\R$ as an orthogonal sum $N_\R=N_{\tau,\R}\oplus N^\tau_\R$. We denote the orthogonal projections onto the factors of this decomposition by $\pr_\tau$ and $\pr^\tau$.

As we will see below, given a normal complex $C_{\Sigma,*}(z)$ and a cone $\tau\in\Sigma$, we will associate a face $\F^\tau(C_{\Sigma,*}(z))$, and this face will lie in the space $N_\R^\tau$. In order to help the reader digest the construction of $\F^\tau(C_{\Sigma,*}(z))$ and its subsequent interpretation as a normal complex, we henceforth make the convention that \emph{$\tau$ superscripts will be used exclusively for objects associated to the vector space $N_\R^\tau$}. For example, $\Sigma^\tau$ will denote a fan in $N_\R^\tau$ and $*^\tau$ will denote an inner product on $N_\R^\tau$.

\subsection{Faces of normal complexes}

There are two primary steps in the face construction for normal complexes. The first step is completely analogous to the polytope setting: we intersect the normal complex with a collection of supporting hyperplanes to obtain a subcomplex. However, in order to view this resulting subcomplex as a normal complex itself, the second step of the construction requires us to translate this polytopal subcomplex to the origin, where we can then endow it with the structure of a normal complex inside $N_\R^\tau$.

Let $\Sigma\subseteq N_\R$ be a simplicial $d$-fan, $*\in\Inn(N_\R)$ an inner product, and $z\in\oCub(\Sigma,*)$ a pseudocubical value. For each cone $\tau\in\Sigma$, define the \textbf{neighborhood of $\tau$ in $\Sigma$} by
\[
\N_\tau\Sigma\defeq \{\pi \mid \pi\preceq \sigma \text{ for some }\sigma\in\Sigma \text{ with }\tau\preceq\sigma\}.
\]
To illustrate this definition, we have darkened the neighborhood of the ray $\rho$ in the following two-dimensional fan.
\begin{center}
\tdplotsetmaincoords{68}{55}
\begin{tikzpicture}[scale=2,tdplot_main_coords]
\draw[draw=blue!20,fill=blue!20,fill opacity=0.8]  (0,0, 0)-- (1, 0, 0) -- (1, 1, 1) -- cycle;
\draw[draw=blue!20,fill=blue!20,fill opacity=0.8]  (0,0, 0)-- (0, 1, 0) -- (1, 1, 1) -- cycle;
\draw[draw=blue!20,fill=blue!20,fill opacity=0.8]  (0,0, 0)-- (0, 0, 1) -- (1, 1, 1) -- cycle;
\draw[draw=blue!0,fill=blue!20,fill opacity=0.2] (0,0, 0)-- (1, 0, 0) -- (0, -1, -1) -- cycle;
\draw[draw=blue!0,fill=blue!20,fill opacity=0.2] (0,0, 0)-- (0, 1, 0) -- (-1, 0, -1) -- cycle;
\draw[draw=blue!0,fill=blue!20,fill opacity=0.2] (0,0, 0)-- (0, 0, 1) -- (-1, -1, 0) -- cycle;
\draw[draw=blue!0,fill=blue!20,fill opacity=0.2] (0,0, 0)-- (-1, -1, 0) -- (-1, -1, -1) -- cycle;
\draw[draw=blue!0,fill=blue!20,fill opacity=0.2] (0,0, 0)-- (-1, 0, -1) -- (-1, -1, -1) -- cycle;
\draw[draw=blue!0,fill=blue!20,fill opacity=0.2] (0,0, 0)-- (0, -1, -1) -- (-1, -1, -1) -- cycle;
\draw[->] (0,0,0) -- (1,0,0);
\draw[->,gray] (0,0,0) -- (0,1,0); 
\draw[->] (0,0,0) -- (0,0,1);
\draw[->] (0,0,0) -- (1,1,1);
\node[right] at (1,1,1) {$\rho$};
\draw[->] (0,0,0) -- (-1,-1,-1);
\draw[->] (0,0,0) -- (-1,-1,0);
\draw[->,gray] (0,0,0) -- (-1,0,-1); 
\draw[->] (0,0,0) -- (0,-1,-1); 
\draw[draw=blue!20] (0, 1, 0) -- (-.41, .59, -.41);
\end{tikzpicture}
\end{center}
Notice that $\N_\tau\Sigma$ is, itself, a simplicial $d$-fan in $N_\R$ whose cones are a subset of $\Sigma$, and the maximal cones of $\N_\tau\Sigma$ comprise all of the maximal cones of $\Sigma$ that contain $\tau$. Since every maximal cone $\sigma\in\N_\tau\Sigma(d)$ contains $\tau$ as a face, it follows from the definitions that each hyperplane $H_{\rho,*}(z)$ with $\rho\in\tau(1)$ is a supporting hyperplane of $P_{\sigma,*}(z)$:
\[
P_{\sigma,*}(z)\subseteq H_{\rho,*}^-(z)\;\;\;\text{ for all }\;\;\; \sigma\in\N_\tau\Sigma(d)\;\;\;\text{ and }\;\;\;\rho\in\tau(1).
\]
Thus, for each $\sigma\in\N_\tau\Sigma(d)$, we obtain a face of $P_{\sigma,*}(z)$ by intersecting with all of these hyperplanes:
\[
\F_\tau(P_{\sigma,*}(z))\defeq P_{\sigma,*}(z)\cap\bigcap_{\rho\in\tau(1)}H_{\rho,*}(z).
\]
The collection of these polytopes along with all of their faces forms a polytopal subcomplex of $C_{\Sigma,*}(z)$, which we denote
\[
\F_\tau(C_{\Sigma,*}(z))\defeq \bigcup_{\sigma\in\N_\tau\Sigma(d)}\widehat{\F_\tau(P_{\sigma,*}(z))}.
\]

To illustrate how the polytopal subcomplex $\F_\tau(C_{\Sigma,*}(z))$ is constructed in a concrete example, the following image depicts a two-dimensional normal complex where we have darkened the collection of maximal polytopes associated to the neighborhood of a ray $\rho$. We have also drawn in the hyperplane associated to $\rho$. The intersection of the hyperplane and the darkened polytopes is $F_\rho(C_{\Sigma,*}(z))$, which, in this example, is a polytopal complex comprised of three line segments meeting at the point $w_{\rho,*}(z)$.
\begin{center}
\tdplotsetmaincoords{68}{55}
\begin{tikzpicture}[scale=1.1,tdplot_main_coords]
\draw[draw=green!20, fill=green!20, fill opacity=1] (0,0,0) -- (0, 0, 1.6) -- (1, 1, 1.6) -- (1.2, 1.2, 1.2) -- cycle; 
\draw[draw=green!20, fill=green!20, fill opacity=1] (0,0,0) -- (0, 1.6, 0) -- (1, 1.6, 1) -- (1.2, 1.2, 1.2) -- cycle; 
\draw[draw=green!20, fill=green!20, fill opacity=1] (0,0,0) -- (1.6, 0, 0) -- (1.6, 1, 1) -- (1.2, 1.2, 1.2) -- cycle; 
\draw[draw=green!20, fill=green!20, fill opacity=.2] (0,0,0) -- (0, 0, 1.6) -- (-1.6, -1.6, 1.6) -- (-1.6, -1.6, 0) -- cycle; 
\draw[draw=green!20, fill=green!20, fill opacity=.2] (0,0,0) -- (0, 1.6, 0) -- (-1.6, 1.6, -1.6) -- (-1.6, 0, -1.6) -- cycle; 
\draw[draw=green!20, fill=green!20, fill opacity=.2] (0,0,0) -- (1.6, 0, 0) -- (1.6, -1.6, -1.6) -- (0, -1.6, -1.6) -- cycle; 
\draw[draw=green!20, fill=green!20, fill opacity=.2] (0,0,0) -- (-1.6, -1.6, 0) -- (-1.6, -1.6, -.4) -- (-1.2, -1.2, -1.2) -- cycle; 
\draw[draw=green!20, fill=green!20, fill opacity=.2] (0,0,0) -- (-1.6, 0, -1.6) -- (-1.6, -.4, -1.6) -- (-1.2, -1.2, -1.2) -- cycle; 
\draw[draw=green!20, fill=green!20, fill opacity=.2] (0,0,0) -- (0, -1.6, -1.6) -- (-.4, -1.6, -1.6) -- (-1.2, -1.2, -1.2) -- cycle; 

\draw[draw=purple!20, fill=purple!20, fill opacity=.8] (2.1, -.9, 2.4) -- (-.9, 2.1, 2.4) -- (0.1, 3.1, 0.4) -- (3.1, 0.1, 0.4) -- cycle;
\node[] at (1.5,1.3,-1.4) {$H_{\rho,*}(z)$};
\draw[thick,->] (1.4,1.5,-1.2) to [bend right=30] (1.5,1.2,-.1);

\draw[thick] (0, 0, 1.6) -- (1, 1, 1.6) -- (1.2, 1.2, 1.2);
\draw[dashed] (0, 1.6, 0) -- (1, 1.6, 1) -- (1.2, 1.2, 1.2);
\draw[thick] (.7, 1.6, .7) -- (1, 1.6, 1) -- (1.2, 1.2, 1.2);
\draw[thick] (1.6, 0, 0) -- (1.6, 1, 1) -- (1.2, 1.2, 1.2);
\draw[] (0, 0, 1.6)  -- (-1.6,  -1.6, 1.6)  -- (-1.6, -1.6, 0);
\draw[dashed] (0, 1.6, 0) -- (-1.6,  1.6, -1.6) -- (-1.6, 0, -1.6);
\draw[] (1.6, 0, 0) -- (1.6, -1.6, -1.6) -- (0, -1.6, -1.6);
\draw[] (-1.6, -1.6, 0) -- (-1.6, -1.6, -.4) -- (-1.2, -1.2, -1.2);
\draw[dashed] (-1.6, 0, -1.6) -- (-1.6, -.4, -1.6) -- (-1.2, -1.2, -1.2);
\draw[] (0, -1.6, -1.6) -- (-.4, -1.6, -1.6) -- (-1.2, -1.2, -1.2);

\draw[line width=.8mm] (1, 1, 1.6) -- (1.2, 1.2, 1.2);
\draw[line width=.8mm]  (1, 1.6, 1) -- (1.2, 1.2, 1.2);
\draw[line width=.8mm](1.6, 1, 1) -- (1.2, 1.2, 1.2);

\draw[] (0, 0, 0) -- (1.2, 1.2, 1.2);
\draw[] (0, 0, 0) -- (0, 0, 1.6);
\draw[dashed] (0, 0, 0) -- (0, 1.6, 0);
\draw[] (0, 0, 0) -- (1.6, 0, 0);
\draw[] (0, 0, 0) -- (0, -1.6, -1.6);
\draw[dashed] (0, 0, 0) -- (-1.6, 0, -1.6);
\draw[] (0, 0, 0) -- (-1.6, -1.6, 0);
\draw[] (0, 0, 0) -- (-1.2, -1.2, -1.2);

\draw[->] (0,0,0) -- (2,2,2);
\node[right] at (2,2,2) {$\rho$};

\node[] at (2.8,2.8,1) {$F_\rho(C_{\Sigma,*}(z))$};
\draw[thick,->] (2.1,2.1,1) to [bend right=30] (1.4,1.4,1.2);
\end{tikzpicture}\end{center}

One might be tempted to call $\F_\tau(C_{\Sigma,*}(z))$ a ``face'' of $C_{\Sigma,*}(z)$; however, a drawback would be that $\F_\tau(C_{\Sigma,*}(z))$ is not, itself, a normal complex (all normal complexes contain the origin, for example, while $\F_\tau(C_{\Sigma,*}(z))$ generally does not). Thus, our construction involves one more step, which is to translate $\F_\tau(C_{\Sigma,*}(z))$ by the vector $w_{\tau,*}(z)$. Notice that, tracking back through the definitions, there is an identification of affine subspaces 
\[
\bigcap_{\rho\in\tau(1)}H_{\rho,*}(z)=N_{\R}^\tau+w_{\tau,*}(z).
\]
Since $\F_\tau(C_{\Sigma,*}(z))$ is, by definition, contained in the left-hand side, it follows that its translation by $-w_{\tau,*}(z)$ is a polytopal complex in $N_\R^\tau$. We define the \textbf{face of $C_{\Sigma,*}(z)$ associated to $\tau\in\Sigma$} to be this polytopal complex:
\[
\F^\tau(C_{\Sigma,*}(z))\defeq \F_\tau(C_{\Sigma,*}(z))-w_{\tau,*}(z)\subseteq N_\R^\tau.
\]
The face associated to the ray $\rho$ in our running example is depicted below inside $N_\R^\rho$.
\begin{center}
\tdplotsetmaincoords{68}{55}
\begin{tikzpicture}[scale=1.1,tdplot_main_coords]
\draw[draw=green!20, fill=green!20, fill opacity=.2] (0,0,0) -- (0, 0, 1.6) -- (1, 1, 1.6) -- (1.2, 1.2, 1.2) -- cycle; 
\draw[draw=green!20, fill=green!20, fill opacity=.2] (0,0,0) -- (0, 1.6, 0) -- (1, 1.6, 1) -- (1.2, 1.2, 1.2) -- cycle; 
\draw[draw=green!20, fill=green!20, fill opacity=.2] (0,0,0) -- (1.6, 0, 0) -- (1.6, 1, 1) -- (1.2, 1.2, 1.2) -- cycle; 
\draw[draw=green!20, fill=green!20, fill opacity=.2] (0,0,0) -- (0, 0, 1.6) -- (-1.6, -1.6, 1.6) -- (-1.6, -1.6, 0) -- cycle; 
\draw[draw=green!20, fill=green!20, fill opacity=.2] (0,0,0) -- (0, 1.6, 0) -- (-1.6, 1.6, -1.6) -- (-1.6, 0, -1.6) -- cycle; 
\draw[draw=green!20, fill=green!20, fill opacity=.2] (0,0,0) -- (1.6, 0, 0) -- (1.6, -1.6, -1.6) -- (0, -1.6, -1.6) -- cycle; 
\draw[draw=green!20, fill=green!20, fill opacity=.2] (0,0,0) -- (-1.6, -1.6, 0) -- (-1.6, -1.6, -.4) -- (-1.2, -1.2, -1.2) -- cycle; 
\draw[draw=green!20, fill=green!20, fill opacity=.2] (0,0,0) -- (-1.6, 0, -1.6) -- (-1.6, -.4, -1.6) -- (-1.2, -1.2, -1.2) -- cycle; 
\draw[draw=green!20, fill=green!20, fill opacity=.2] (0,0,0) -- (0, -1.6, -1.6) -- (-.4, -1.6, -1.6) -- (-1.2, -1.2, -1.2) -- cycle; 

\draw[line width=.01mm] (0, 0, 1.6) -- (1, 1, 1.6) -- (1.2, 1.2, 1.2);
\draw[line width=.01mm,dashed] (0, 1.6, 0) -- (1, 1.6, 1) -- (1.2, 1.2, 1.2);
\draw[line width=.01mm] (.7, 1.6, .7) -- (1, 1.6, 1) -- (1.2, 1.2, 1.2);
\draw[line width=.01mm] (1.6, 0, 0) -- (1.6, 1, 1) -- (1.2, 1.2, 1.2);
\draw[line width=.01mm] (0, 0, 1.6)  -- (-1.6,  -1.6, 1.6)  -- (-1.6, -1.6, 0);
\draw[line width=.01mm,dashed] (0, 1.6, 0) -- (-1.6,  1.6, -1.6) -- (-1.6, 0, -1.6);
\draw[line width=.01mm] (1.6, 0, 0) -- (1.6, -1.6, -1.6) -- (0, -1.6, -1.6);
\draw[line width=.01mm] (-1.6, -1.6, 0) -- (-1.6, -1.6, -.4) -- (-1.2, -1.2, -1.2);
\draw[line width=.01mm,dashed] (-1.6, 0, -1.6) -- (-1.6, -.4, -1.6) -- (-1.2, -1.2, -1.2);
\draw[line width=.01mm] (0, -1.6, -1.6) -- (-.4, -1.6, -1.6) -- (-1.2, -1.2, -1.2);

\draw[draw=purple!20, fill=purple!20, fill opacity=.8] (.9, -2.1, 1.2) -- (-2.1, .9, 1.2) -- (-1.1, 1.9, -.8) -- (1.9, -1.1, -.8) -- cycle;
\node[] at (1.5,1.3,-2.5) {$N_\R^\rho=H_{\rho,*}(z)-w_{\rho,*}(z)$};
\draw[thick,->] (.2,.3,-2.4) to [bend right=30] (.3,0,-1.2);

\draw[line width=.8mm] (-.2, -.2, .4) -- (0, 0, 0);
\draw[line width=.8mm]  (-.2, .4, -.2) -- (0, 0, 0);
\draw[line width=.8mm](.4, -.2, -.2) -- (0, 0, 0);

\draw[line width=.01mm] (0, 0, 0) -- (1.2, 1.2, 1.2);
\draw[line width=.01mm] (0, 0, 0) -- (0, 0, 1.6);
\draw[line width=.01mm,dashed] (0, 0, 0) -- (0, 1.6, 0);
\draw[line width=.01mm] (0, 0, 0) -- (1.6, 0, 0);
\draw[line width=.01mm] (0, 0, 0) -- (0, -1.6, -1.6);
\draw[line width=.01mm,dashed] (0, 0, 0) -- (-1.6, 0, -1.6);
\draw[line width=.01mm] (0, 0, 0) -- (-1.6, -1.6, 0);
\draw[line width=.01mm] (0, 0, 0) -- (-1.2, -1.2, -1.2);

\draw[->] (0,0,0) -- (2,2,2);
\node[right] at (2,2,2) {$\rho$};

\node[] at (3,3,0) {$F^\rho(C_{\Sigma,*}(z))=F_\rho(C_{\Sigma,*}(z))-w_{\rho,*}(z)$};
\draw[thick,->] (.9,.9,-.2) to [bend right=30] (.2,.2,0);
\end{tikzpicture}\end{center}

The next pair of images depicts the subcomplex $\F_\rho(C_{\Sigma,*}(z))\subseteq C_{\Sigma,*}(z)$ and, after translating to the origin, the face $\F^\rho(C_{\Sigma,*}(z))$, where $\rho$ is a ray of a three-dimensional fan.

\begin{center}
\begin{tikzpicture}
\node[] at (0,0) {\includegraphics[scale=.3]{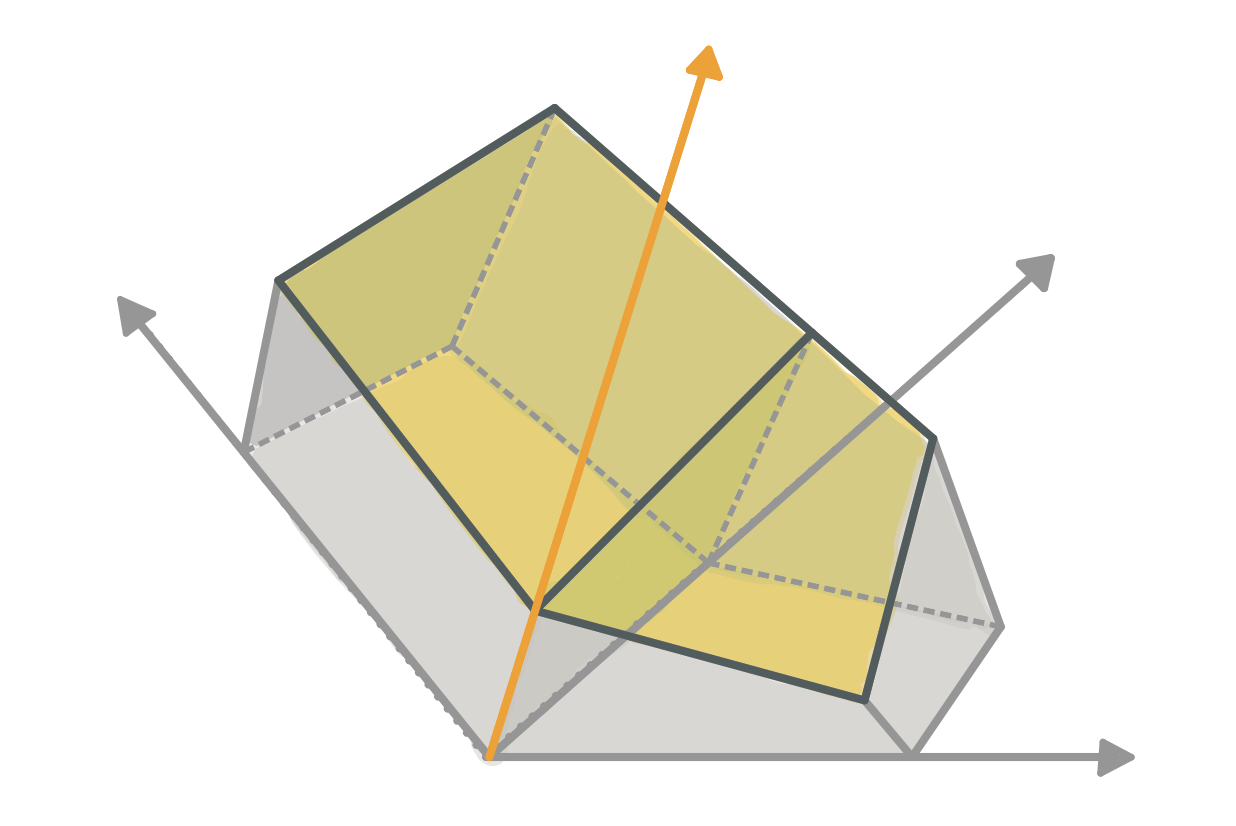}};
\node[] at (8,0) {\includegraphics[scale=.25]{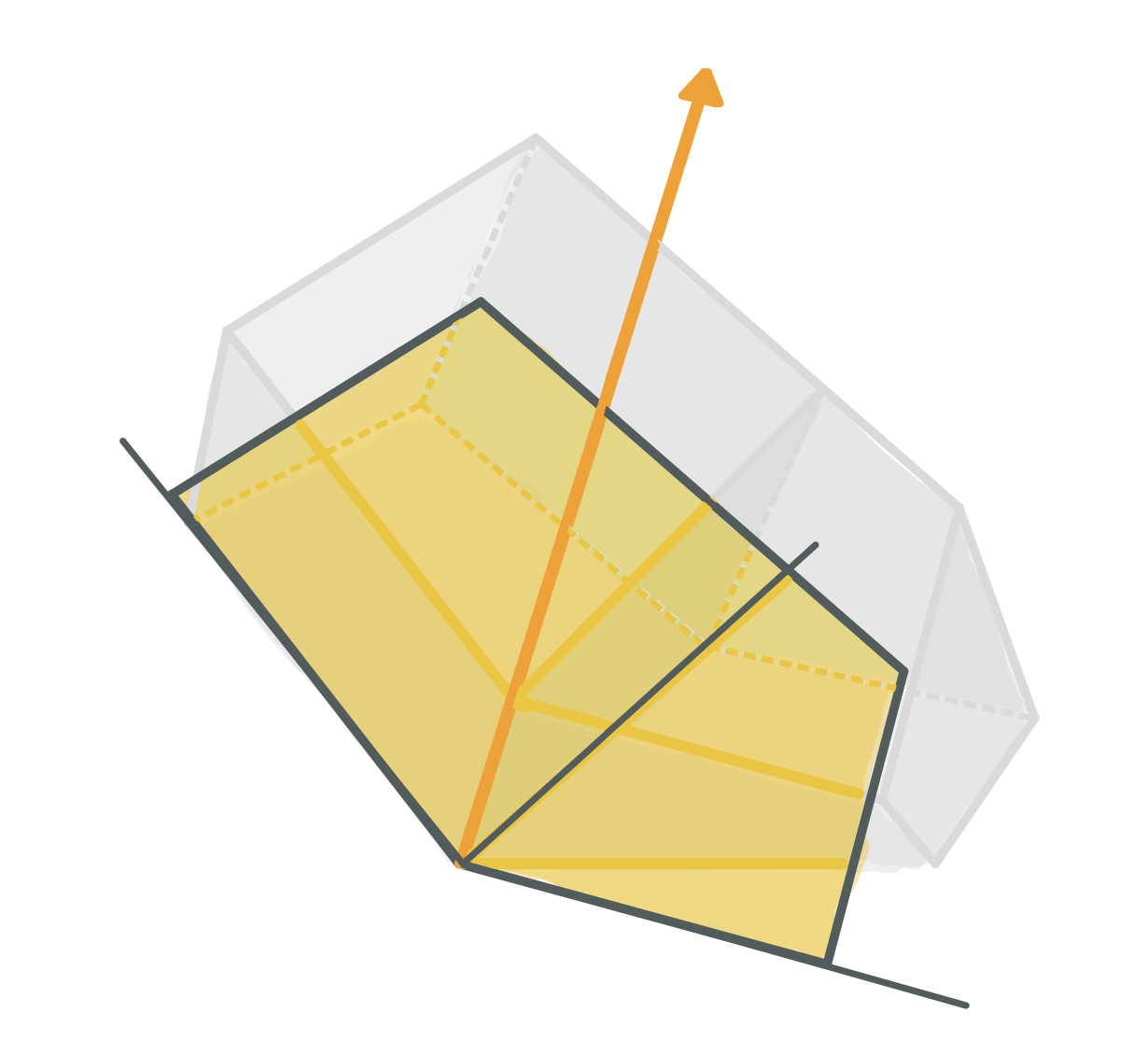}};
\node[] at (.75,1.7) {$\rho$};
\node[] at (8.8,1.8) {$\rho$};
\node[] at (-2.8,1.5) {$F_\rho(C_{\Sigma,*}(z))$};
\draw[thick,->] (-1.7,1.5) to [bend left=30] (-.7,.9);
\node[] at (5.3,-1) {$F^\rho(C_{\Sigma,*}(z))$};
\draw[thick,->] (6.4,-1) to [bend right=20] (7.5,-.5);
\end{tikzpicture}
\end{center}

In the following subsections, it will also be useful to have notation for translates of the polytopes $\F_\tau(P_{\sigma,*}(z))$. We define
\[
\F^\tau(P_{\sigma,*}(z))\defeq \F_\tau(P_{\sigma,*}(z))-w_{\tau,*}(z).
\]
In terms of these translated polytopes, we can write the $\tau$-face of $C_{\Sigma,*}(z)$ as
\[
\F^\tau(C_{\Sigma,*}(z))=\bigcup_{\sigma\in\N_\tau\Sigma(d)}\widehat{\F^\tau(P_{\sigma,*}(z))}.
\]

\subsection{Faces as normal complexes}

Our aim in this subsection is to realize each face $\F^\tau(C_{\Sigma,*}(z))$ as a normal complex. In order to do so, we require several ingredients; namely, we require a marked, pure, simplicial fan $\Sigma^\tau$ in $N_\R^\tau$, an inner product $*^\tau$ on $N_\R^\tau$, and a pseudocubical value $z^\tau\in\oCub(\Sigma^\tau,*^\tau)$. We now define each of these ingredients.

For each cone $\tau\in\Sigma$, define the \textbf{star of $\Sigma$ at $\tau\in\Sigma$} to be the fan in $N_\R^\tau$ comprised of all projections of cones in the neighborhood of $\tau$:
\[
\Sigma^\tau\defeq\{\pr^\tau(\pi)\mid \pi\in\N_\tau\Sigma\}.
\]
The star of a two-dimensional fan $\Sigma$ at a ray $\rho$ is depicted below. In the image, there are three two-dimensional cones in the neighborhood of $\rho$ that are projected onto three one-dimensional cones that comprise the maximal cones in the star fan $\Sigma^\rho$.
\begin{center}
\tdplotsetmaincoords{68}{55}
\begin{tikzpicture}[scale=2,tdplot_main_coords]
\draw[draw=blue!20,fill=blue!20,fill opacity=0.8]  (0,0, 0)-- (1, 0, 0) -- (1, 1, 1) -- cycle;
\draw[draw=blue!20,fill=blue!20,fill opacity=0.8]  (0,0, 0)-- (0, 1, 0) -- (1, 1, 1) -- cycle;
\draw[draw=blue!20,fill=blue!20,fill opacity=0.8]  (0,0, 0)-- (0, 0, 1) -- (1, 1, 1) -- cycle;
\draw[draw=blue!0,fill=blue!20,fill opacity=0.2] (0,0, 0)-- (1, 0, 0) -- (0, -1, -1) -- cycle;
\draw[draw=blue!0,fill=blue!20,fill opacity=0.2] (0,0, 0)-- (0, 1, 0) -- (-1, 0, -1) -- cycle;
\draw[draw=blue!0,fill=blue!20,fill opacity=0.2] (0,0, 0)-- (0, 0, 1) -- (-1, -1, 0) -- cycle;
\draw[draw=blue!0,fill=blue!20,fill opacity=0.2] (0,0, 0)-- (-1, -1, 0) -- (-1, -1, -1) -- cycle;
\draw[draw=blue!0,fill=blue!20,fill opacity=0.2] (0,0, 0)-- (-1, 0, -1) -- (-1, -1, -1) -- cycle;
\draw[draw=blue!0,fill=blue!20,fill opacity=0.2] (0,0, 0)-- (0, -1, -1) -- (-1, -1, -1) -- cycle;
\draw[->] (0,0,0) -- (1,0,0);
\draw[->,gray] (0,0,0) -- (0,1,0); 
\draw[->] (0,0,0) -- (0,0,1);
\draw[->] (0,0,0) -- (1,1,1);
\node[right] at (1,1,1) {$\rho$};
\draw[->] (0,0,0) -- (-1,-1,-1);
\draw[->] (0,0,0) -- (-1,-1,0);
\draw[->,gray] (0,0,0) -- (-1,0,-1); 
\draw[->] (0,0,0) -- (0,-1,-1); 
\draw[draw=blue!20] (0, 1, 0) -- (-.41, .59, -.41);
\node[] at (-.8,-.8,.6) {$\Sigma\subseteq N_\R$};
\end{tikzpicture}
\hspace{50bp}
\begin{tikzpicture}[scale=2,tdplot_main_coords]
\draw[draw=purple!20, fill=purple!20, fill opacity=.8] (.9*.65, -2.1*.65, 1.2*.65) -- (-2.1*.65, .9*.65, 1.2*.65) -- (-1.1*.65, 1.9*.65, -.8*.65) -- (1.9*.65, -1.1*.65, -.8*.65) -- cycle;
\node[] at (0.6,0.6,0.5) {$\Sigma^\rho\subseteq N_\R^\rho$};
\draw[thick,->] (0.3,0.3,0.49) to [bend right=20] (-.05,-.05,.15);
\draw[line width=.4mm,->] (0,0,0) -- (.66,-.33,-.33);
\draw[line width=.4mm,->] (0,0,0) -- (-.33,.66,-.33); 
\draw[line width=.4mm,->] (0,0,0) -- (-.33,-.33,.66);
\end{tikzpicture}
\end{center}
Henceforth, we use the shorthand $\pi^\tau = \pr^\tau(\pi)$. 

Given any cone $\pi^\tau\in\Sigma^\tau$ with $\pi\in\N_\tau\Sigma$, we can also view $\pi^\tau$ as the projection of the larger cone $\sigma=\pi\cup\tau\in\N_\tau\Sigma$. Note that $\sigma$ is the unique maximal cone in $\N_\tau\Sigma$ that projects onto $\pi^\tau$, from which it follows that each cone in $\Sigma^\tau$ is the projection of a \emph{distinguished} cone in $\N_\tau\Sigma$. In other words, there is a bijection
\begin{align*}
\{\sigma\in\N_\tau\Sigma \mid \tau\preceq\sigma\}&\rightarrow\Sigma^\tau\\
\sigma&\mapsto \sigma^\tau.
\end{align*}
From the assumptions that $\Sigma$ is a simplicial $d$-fan, it follow that $\Sigma^\tau$ is a simplicial fan in $N^\tau_\R$ that is pure of dimension $d^\tau = d-\dim(\tau)$.  Moreover, the simplicial hypothesis on $\Sigma$ implies that each ray $\eta\in\Sigma^\tau(1)$ is the projection of a unique ray $\hat\eta\in\N_\tau\Sigma(1)$, and we can use this to mark each ray $\eta\in\Sigma^\tau(1)$ with the vector $\pr^\tau(u_{\hat\eta})$.

We now have a marked, pure, simplicial fan in $N_\R^\tau$, so it remains to define an inner product and pseudocubical value. The inner product $*^\tau\in\Inn(N_\R^\tau)$ is simply defined as the restriction of the inner product $*\in\Inn(N_\R)$ to the subspace $N_\R^\tau$. Lastly, given any $z\in\R^{\Sigma(1)}$, we define $z^\tau\in\R^{\Sigma^\tau(1)}$ by the rule
\[
z^\tau_{\eta}=z_{\hat\eta}-w_{\tau,*}(z)*u_{\hat\eta},
\]
where, as before, $\hat\eta\in\N_\tau\Sigma(1)$ is the unique ray with $\pr^\tau(\hat\eta)=\eta$. 

We now have all the ingredients necessary to state and prove the following result, which asserts that faces of normal complexes are, themselves, normal complexes.

\begin{proposition}\label{prop:facesarenormalcomplexes}
Let $\Sigma\subseteq N_\R$ be a simplicial $d$-fan, $*\in\Inn(N_\R)$ an inner product, and $\tau\in\Sigma$ a cone. If $z\in\R^{\Sigma(1)}$ is (pseudo)cubical with respect to $(\Sigma,*)$, then $z^\tau$ is (pseudo)cubical with respect to $(\Sigma^\tau,*^\tau)$ and
\[
\F^\tau(C_{\Sigma,*}(z))=C_{\Sigma^\tau,*^\tau}(z^\tau).
\]
\end{proposition}

We note that the first statement---that $z^\tau$ is (pseudo)cubical---is necessary in order for $C_{\Sigma^\tau,*^\tau}(z^\tau)$ to even be well-defined. Proposition~\ref{prop:facesarenormalcomplexes} is a statement about normal complexes, or equivalently, about the polytopes that comprise those complexes. In order to prove Proposition~\ref{prop:facesarenormalcomplexes}, we first prove the following key lemma, which concerns just the vertices of the polytopes.

\begin{lemma}\label{lemma:projectingws}
Let $\Sigma\subseteq N_\R$ be a simplicial $d$-fan, $*\in\Inn(N_\R)$ an inner product, and $\tau\in\Sigma$ a cone. For any $\sigma\in\Sigma$ with $\tau\preceq\sigma$, we have
\[
\pr^\tau(w_{\sigma,*}(z))=w_{\sigma,*}(z)-w_{\tau,*}(z)=w_{\sigma^\tau,*^\tau}(z^\tau).
\]
\end{lemma}

\begin{proof}
We start by establishing the first equality. To do so, we begin by arguing that $w_{\sigma,*}(z)-w_{\tau,*}(z)\in N_\R^\tau$. Since $N_\R^\tau=N_{\tau,\R}^\perp$, it suffices to prove that $w_{\sigma,*}(z)-w_{\tau,*}(z)$ is orthogonal to the basis $\{u_\rho\mid \rho\in\tau(1)\}\subseteq N_{\tau,\R}$. By definition of the $w$ vectors and the assumption that $\tau\preceq\sigma$, we compute
\[
(w_{\sigma,*}(z)-w_{\tau,*}(z))*u_\rho=z_\rho-z_\rho=0\;\;\;\text{ for all }\;\;\;\rho\in\tau(1),
\]
from which it follows that $w_{\sigma,*}(z)-w_{\tau,*}(z)\in N_\R^\tau$. Since $N_\R=N_{\tau,\R}\oplus N_\R^\tau$, the orthogonal decomposition $w_{\sigma,*}(z)=w_{\tau,*}(z)\,+\,(w_{\sigma,*}(z)-w_{\tau,*}(z))$ then implies that 
\begin{equation}\label{eq:projectingvertices}
\pr_\tau(w_{\sigma,*}(z))=w_{\tau,*}(z)\;\;\;\text{ and }\;\;\;\pr^\tau(w_{\sigma,*}(z))=w_{\sigma,*}(z)-w_{\tau,*}(z).
\end{equation}

To prove that $w_{\sigma,*}(z)-w_{\tau,*}(z)=w_{\sigma^\tau,*^\tau}(z^\tau)$, we now argue that $w_{\sigma,*}(z)-w_{\tau,*}(z)$ is an element of $N_{\sigma^\tau,\R}$ and is a solution of the equations defining $w_{\sigma^\tau,*^\tau}(z^\tau)$:
\begin{equation}\label{eq:definingw2}
v*^\tau u_\eta=z_\eta^\tau\;\;\;\text{ for all }\;\;\;\eta\in\sigma^\tau(1).
\end{equation}
To check that $w_{\sigma,*}(z)-w_{\tau,*}(z)\in N_{\sigma^\tau,\R}$, we start by observing that we can write
\[
w_{\sigma,*}(z)=\sum_{\rho\in\sigma(1)}a_{\rho}\;u_\rho
\]
for some values $a_{\rho}\in\R$, in which case
\begin{align*}
w_{\sigma,*}(z)-w_{\tau,*}(z)&=\pr^\tau(w_{\sigma,*}(z))\\
&=\sum_{\rho\in\sigma(1)\setminus\tau(1)}a_{\rho}\;\pr^\tau(u_\rho)\\
&=\sum_{\eta\in\sigma^\tau(1)}a_{\hat\eta}\;u_{\eta},
\end{align*}
where the first equality uses \eqref{eq:projectingvertices}, the second uses that $\pr^\tau$ vanishes on $N_{\tau,\R}$, and the third uses that the rays of $\sigma^\tau$ are in natural bijection with $\sigma(1)\setminus\tau(1)$.
Lastly, we peel back the definitions to check that $w_{\sigma,*}(z)-w_{\tau,*}(z)$ is a solution of Equations~\eqref{eq:definingw2}:
\begin{align*}
(w_{\sigma,*}(z)-w_{\tau,*}(z))*^\tau u_\eta&= (w_{\sigma,*}(z)-w_{\tau,*}(z))*(u_{\hat\eta}-\pr_\tau(u_{\hat\eta}))\\
&=w_{\sigma,*}(z)*u_{\hat\eta}\;-\;w_{\tau,*}(z)*u_{\hat\eta}\;-\;\big(w_{\sigma,*}(z)-w_{\tau,*}(z)\big)*\pr_\tau(u_{\hat\eta})\\
&=z_{\hat\eta}-w_{\tau,*}(z)*u_{\hat\eta}\\
&=z_\eta^\tau,
\end{align*}
where the first equality uses the orthogonal decomposition of $u_{\hat\eta}$ and the fact that $*^\tau$ is just the restriction of $*$, the second equality uses linearity of the inner product, and the third equality uses the definition of $w_{\sigma,*}(z)$ along with the fact that the vectors $\pr_\tau(u_{\hat\eta})$ and $w_{\sigma,*}(z)-w_{\tau,*}(z)=\pr^\tau(w_{\sigma,*}(z))$ are in orthogonal subspaces.
\end{proof}

\begin{proof}[Proof of Proposition~\ref{prop:facesarenormalcomplexes}]
To prove the first statement in the cubical setting, assume that $z\in\R^{\Sigma(1)}$ is cubical. This means that, for every $\sigma\in\Sigma$, we can write
\[
w_{\sigma,*}(z)=\sum_{\rho\in\sigma(1)}a_{\rho} u_\rho
\]
for some positive values $a_{\rho}\in\R_{>0}$. Consider any cone of $\Sigma^\tau$, which we can write as $\sigma^\tau$ with $\tau\preceq\sigma$. Applying the lemma, we then see that
\begin{align*}
w_{\sigma^\tau,*^\tau}(z^\tau)&=\pr^\tau(w_{\sigma,*}(z))\\
&=\sum_{\rho\in\sigma(1)\setminus\tau(1)}a_{\rho}\; \pr^\tau(u_\rho)\\
&=\sum_{\eta\in\sigma^\tau(1)}a_{\hat\eta}\; u_\eta.
\end{align*}
This shows that $w_{\sigma^\tau,*^\tau}(z^\tau)$ can be written as a positive combination of the ray generators of $\sigma^\tau$, proving that $z^\tau\in\Cub(\Sigma^\tau,*^\tau)$. The proof  in the pseudocubical setting is identical but with ``positive'' replaced by ``nonnegative.''

To prove that
\[
\F^\tau(C_{\Sigma,*}(z))=C_{\Sigma^\tau,*^\tau}(z^\tau),
\]
it suffices to identify the maximal polytopes in these complexes. In other words, we must prove that, for every $\sigma\in\N_\tau\Sigma(d)$, we have
\begin{equation}\label{eq:polytranslate}
\F^\tau(P_{\sigma,*}(z))=P_{\sigma^\tau,*^\tau}(z^\tau).
\end{equation}
To prove \eqref{eq:polytranslate}, we analyze the vertices of these polytopes.

By Proposition~\ref{prop:normalcomplexprelims}, the vertices of $P_{\sigma,*}(z)$ are $\{w_{\pi,*}(z)\mid \pi\preceq\sigma\}$. Since 
\[
F_\tau(P_{\sigma,*}(z))=P_{\sigma,*}(z)\cap\bigcap_{\rho\in\tau(1)}H_{\rho,*}(z),
\] 
it follows that the vertices of $F_\tau(P_{\sigma,*}(z))$ are 
\[
\{w_{\pi,*}(z)\mid \pi\preceq\sigma\text{ and }w_{\pi,*}(z)*u_\rho=z_\rho\text{ for all }\rho\in\tau(1)\}.
\]
If a cone $\pi\preceq \sigma$ satisfies $w_{\pi,*}(z)*u_\rho=z_\rho$ for all $\rho\in\tau(1)$, then the definition of the $w$-vectors implies that $w_{\pi,*}(z)=w_{\pi\cup\tau,*}(z)$, and it follows that the vertices of $F_\tau(P_{\sigma,*}(z))$ are
\[
\mathrm{Vert}\big(F_\tau(P_{\sigma,*}(z))\big)=\{w_{\pi,*}(z)\mid \tau\preceq\pi\preceq\sigma\}.
\]
Upon translating by $w_{\tau,*}(z)$ to get from $F_\tau(P_{\sigma,*}(z))$ to $F^\tau(P_{\sigma,*}(z))$, we see that
\begin{align*}
\mathrm{Vert}\big(F^\tau(P_{\sigma,*}(z))\big)&=\{w_{\pi,*}(z)-w_{\tau,*}(z)\mid \tau\preceq\pi\preceq\sigma\}\\
&=\{w_{\pi^\tau,*^\tau}(z^\tau)\mid \pi^\tau\preceq\sigma^\tau\}\\
&=\mathrm{Vert}\big(P_{\sigma^\tau,*^\tau}(z^\tau))\big),
\end{align*}
where the second equality is an application of Lemma~\ref{lemma:projectingws} and the third is an application of Proposition~\ref{prop:normalcomplexprelims}. Having matched the vertices of the polytopes in \eqref{eq:polytranslate}, the equality of polytopes then follows.
\end{proof}

The importance of Proposition~\ref{prop:facesarenormalcomplexes} is that it allows us to endow each of the faces of a normal complex with the structure of a normal complex, and in particular, it then allows us to compute (mixed) volumes of faces. More specifically, if $\omega:\Sigma(d)\rightarrow\R_{>0}$ is a weight function, then we obtain a weight function $\omega^\tau:\Sigma^\tau(d^\tau)\rightarrow\R_{>0}$ defined by $\omega^\tau(\sigma^\tau)=\omega(\sigma)$ for all $\sigma\in\Sigma(d)$. The volume of the face $\F^\tau(C_{\Sigma,*}(z))$ weighted by $\omega$ is
\[
\Vol_{\Sigma^\tau,\omega^\tau,*^\tau}(z^\tau).
\]
Similarly, the mixed volume of the faces $\F^\tau(C_{\Sigma,*}(z_1)),\dots \F^\tau(C_{\Sigma,*}(z_{d^\tau}))$ weighted by $\omega$ is
\[
\MVol_{\Sigma^\tau,\omega^\tau,*^\tau}(z_1^\tau,\dots,z_{d^\tau}^\tau).
\]
In the next two subsections, we use these concepts to prove fundamental results relating (mixed) volumes of normal complexes to the (mixed) volumes of their facets. In making arguments using mixed volumes, it will be useful to consider facets of facets; as such, the next result---asserting that the face of a face of a normal complex is a face of the original normal complex---will be useful.

\begin{proposition}\label{prop:faceofaface}
Let $\Sigma\subseteq N_\R$ be a simplicial $d$-fan, $*\in\Inn(N_\R)$ an inner product, and $z\in\oCub(\Sigma,*)$ a pseudocubical value. If $\tau,\pi\in\Sigma$ with $\tau\preceq\pi$, then
\[
\F^{\pi^\tau}(\F^\tau(C_{\Sigma,*}(z)))=\F^\pi(C_{\Sigma,*}(z)).
\]
\end{proposition}

\begin{proof}
By Proposition~\ref{prop:facesarenormalcomplexes}, the claim in this proposition is equivalent to
\[
\F^{\pi^\tau}(C_{\Sigma^\tau,*^\tau}(z^\tau))=C_{\Sigma^\pi,*^\pi}(z^\pi).
\]
It suffices to match the maximal polytopes in these complexes, so we must prove:
\begin{equation}\label{eq:faceofafacepoly}
\F^{\pi^\tau}(P_{\sigma^\tau,*^\tau}(z^\tau))=P_{\sigma^\pi,*^\pi}(z^\pi)\;\;\;\text{ for all }\;\;\;\sigma\in\Sigma(d)\;\;\;\text{ with }\;\;\;\tau\preceq\sigma.
\end{equation}
The vertices of the polytope in the left-hand side of \eqref{eq:faceofafacepoly} are
\[
\{w_{\mu^\tau,*^\tau}(z^\tau)-w_{\pi^\tau,*^\tau}(z^\tau)\mid \pi^\tau\preceq\mu^\tau\preceq \sigma^\tau \}
\]
while the vertices in the right-hand side of \eqref{eq:faceofafacepoly} are
\[
\{w_{\mu^\pi,*^\pi}(z^\pi)\mid \mu^\pi\preceq\sigma^\pi\}.
\]
Notice that both sets of vertices are indexed by $\mu\in\Sigma$ with $\pi\preceq\mu\preceq\sigma$, and we have
\begin{align*}
w_{\mu^\tau,*^\tau}(z^\tau)-w_{\pi^\tau,*^\tau}(z^\tau)&=\pr^{\pi^\tau}(w_{\mu^\tau,*^\tau}(z^\tau))\\
&=\pr^{\pi^\tau}(\pr^\tau(w_{\mu,*}(z)))\\
&=\pr^\pi(w_{\mu,*}(z))\\
&=w_{\mu^\pi,*^\pi}(z^\pi),
\end{align*}
where the first, second, and fourth equalities are Lemma~\ref{lemma:projectingws}, while the second is the observation that the projection $\pr^\pi$ can be broken up into two steps: $\pr^\pi=\pr^{\pi^\tau}\circ\pr^\tau$. Thus, the vertices of the polytopes in \eqref{eq:faceofafacepoly} match up, and the proposition follows.
\end{proof}

\subsection{Volumes and facets}

This subsection is devoted to proving the following result, which relates the volume of a normal complex to the volumes of its facets.

\begin{proposition}\label{prop:pyramidvolume}
Let $\Sigma\subseteq N_\R$ be a simplicial $d$-fan with weight function $\omega:\Sigma(d)\rightarrow\R_{>0}$,  let $*\in\Inn(N_\R)$ be an inner product, and let $z\in\oCub(\Sigma,*)$ be a pseudocubical value. Then
\[
\Vol_{\Sigma,\omega,*}(z)=\sum_{\rho\in\Sigma(1)}z_\rho\Vol_{\Sigma^\rho,\omega^\rho,*^\rho}(z^\rho).
\]
\end{proposition}

The sum in the right-hand side of the theorem corresponds to decomposing the normal complex into pyramids over its facets, as depicted in the next image.

\begin{center}
\tdplotsetmaincoords{68}{55}
\begin{tikzpicture}[scale=1.1,tdplot_main_coords]
\draw[draw=green!20, fill=green!20, fill opacity=.8] (0,0,0) -- (0, 1.6, 0) -- (-1.6, 1.6, -1.6) -- (-1.6, 0, -1.6) -- cycle; 
\draw[] (0,0,0) -- (0, 1.6, 0) -- (-1.6, 1.6, -1.6) -- (-1.6, 0, -1.6) -- cycle; 
\draw[draw=green!20, fill=green!20, fill opacity=.8] (0,0,0) -- (0, 1.6, 0) -- (1, 1.6, 1) -- (1.2, 1.2, 1.2) -- cycle; 
\draw[] (0,0,0) -- (0, 1.6, 0) -- (1, 1.6, 1) -- (1.2, 1.2, 1.2) -- cycle; 
\draw[draw=green!20, fill=green!20, fill opacity=.8] (0,0,0) -- (0, 0, 1.6) -- (1, 1, 1.6) -- (1.2, 1.2, 1.2) -- cycle;
\draw[] (0,0,0) -- (0, 0, 1.6) -- (1, 1, 1.6) -- (1.2, 1.2, 1.2) -- cycle;
\draw[draw=green!20, fill=green!20, fill opacity=.8] (0,0,0) -- (1.6, 0, 0) -- (1.6, 1, 1) -- (1.2, 1.2, 1.2) -- cycle;
\draw[] (0,0,0) -- (1.6, 0, 0) -- (1.6, 1, 1) -- (1.2, 1.2, 1.2) -- cycle;
\draw[draw=green!20, fill=green!20, fill opacity=.8] (0,0,0) -- (0, 0, 1.6) -- (-1.6, -1.6, 1.6) -- (-1.6, -1.6, 0) -- cycle; 
\draw[] (0,0,0) -- (0, 0, 1.6) -- (-1.6, -1.6, 1.6) -- (-1.6, -1.6, 0) -- cycle; 
\draw[draw=green!20, fill=green!20, fill opacity=.8] (0,0,0) -- (1.6, 0, 0) -- (1.6, -1.6, -1.6) -- (0, -1.6, -1.6) -- cycle;
\draw[]  (0,0,0) -- (1.6, 0, 0) -- (1.6, -1.6, -1.6) -- (0, -1.6, -1.6) -- cycle;
\draw[draw=green!20, fill=green!20, fill opacity=.8] (0,0,0) -- (-1.6, -1.6, 0) -- (-1.6, -1.6, -.4) -- (-1.2, -1.2, -1.2) -- cycle;
\draw[] (0,0,0) -- (-1.6, -1.6, 0) -- (-1.6, -1.6, -.4) -- (-1.2, -1.2, -1.2) -- cycle;
\draw[draw=green!20, fill=green!20, fill opacity=.8] (0,0,0) -- (-1.6, 0, -1.6) -- (-1.6, -.4, -1.6) -- (-1.2, -1.2, -1.2) -- cycle; 
\draw[] (0,0,0) -- (-1.6, 0, -1.6) -- (-1.6, -.4, -1.6) -- (-1.2, -1.2, -1.2) -- cycle; 
\draw[draw=green!20, fill=green!20, fill opacity=.8] (0,0,0) -- (0, -1.6, -1.6) -- (-.4, -1.6, -1.6) -- (-1.2, -1.2, -1.2) -- cycle; 
\draw[] (0,0,0) -- (0, -1.6, -1.6) -- (-.4, -1.6, -1.6) -- (-1.2, -1.2, -1.2) -- cycle; 
\end{tikzpicture}
\hspace{50bp}
\begin{tikzpicture}[scale=.85,tdplot_main_coords]

\draw[draw=green!20, fill=green!20, fill opacity=.8] (0,0.5,0) -- (0, 2.1, 0) -- (-1.6, 2.1, -1.6) -- cycle; 
\draw[] (0,0.5,0) -- (0, 2.1, 0) -- (-1.6, 2.1, -1.6) -- cycle; 
\draw[draw=green!20, fill=green!20, fill opacity=.8] (0,0.5,0) -- (0, 2.1, 0) -- (1, 2.1, 1) -- cycle; 
\draw[] (0,0.5,0) -- (0, 2.1, 0) -- (1, 2.1, 1) -- cycle; 

\draw[draw=green!20, fill=green!20, fill opacity=.8] (.5,.5,.5) -- (1.5, 1.5, 2.1) -- (1.7, 1.7, 1.7) -- cycle; 
\draw[] (0.5,0.5,0.5) -- (1.5, 1.5, 2.1) -- (1.7, 1.7, 1.7) -- cycle; 
\draw[draw=green!20, fill=green!20, fill opacity=.8] (0.5,0.5,0.5) -- (1.5, 2.1, 1.5) -- (1.7, 1.7, 1.7) -- cycle; 
\draw[] (0.5,0.5,0.5) -- (1.5, 2.1, 1.5) -- (1.7, 1.7, 1.7) -- cycle; 
\draw[draw=green!20, fill=green!20, fill opacity=.8] (0.5,0.5,0.5)  -- (2.1, 1.5, 1.5) -- (1.7, 1.7, 1.7) -- cycle; 
\draw[] (0.5,0.5,0.5)  -- (2.1, 1.5, 1.5) -- (1.7, 1.7, 1.7) -- cycle; 

\draw[draw=green!20, fill=green!20, fill opacity=.8] (0,0,0.5) -- (0, 0, 2.1) -- (1, 1, 2.1)  -- cycle; 
\draw[] (0,0,0.5) -- (0, 0, 2.1) -- (1, 1, 2.1)  -- cycle; 
\draw[draw=green!20, fill=green!20, fill opacity=.8] (0,0,0.5) -- (0, 0, 2.1) -- (-1.6, -1.6, 2.1) -- cycle; 
\draw[] (0,0,0.5) -- (0, 0, 2.1) -- (-1.6, -1.6, 2.1) -- cycle; 

\draw[draw=green!20, fill=green!20, fill opacity=.8] (0.5,0,0) -- (2.1, 0, 0) -- (2.1, 1, 1) -- cycle; 
\draw[] (0.5,0,0) -- (2.1, 0, 0) -- (2.1, 1, 1) -- cycle; 
\draw[draw=green!20, fill=green!20, fill opacity=.8] (0.5,0,0) -- (2.1, 0, 0) -- (2.1, -1.6, -1.6) -- cycle; 
\draw[] (0.5,0,0) -- (2.1, 0, 0) -- (2.1, -1.6, -1.6) -- cycle; 

\draw[draw=green!20, fill=green!20, fill opacity=.8] (-0.5,-0.5,0) -- (-2.1, -2.1, 1.6) -- (-2.1, -2.1, 0) -- cycle; 
\draw[] (-0.5,-0.5,0) -- (-2.1, -2.1, 1.6) -- (-2.1, -2.1, 0) -- cycle; 
\draw[draw=green!20, fill=green!20, fill opacity=.8] (-0.5,-0.5,0) -- (-2.1, -2.1, 0) -- (-2.1, -2.1, -.4) -- cycle; 
\draw[] (-0.5,-0.5,0) -- (-2.1, -2.1, 0) -- (-2.1, -2.1, -.4) -- cycle; 

\draw[draw=green!20, fill=green!20, fill opacity=.8] (-0.5,-0.5,-0.5) -- (-2.1, -2.1, -.9) -- (-1.7, -1.7, -1.7) -- cycle; 
\draw[] (-0.5,-0.5,-0.5) -- (-2.1, -2.1, -.9) -- (-1.7, -1.7, -1.7) -- cycle; 
\draw[draw=green!20, fill=green!20, fill opacity=.8] (-0.5,-0.5,-0.5) -- (-2.1, -.9, -2.1) -- (-1.7, -1.7, -1.7) -- cycle; 
\draw[] (-0.5,-0.5,-0.5) -- (-2.1, -.9, -2.1) -- (-1.7, -1.7, -1.7) -- cycle; 
\draw[draw=green!20, fill=green!20, fill opacity=.8] (-0.5,-0.5,-0.5) -- (-.9, -2.1, -2.1) -- (-1.7, -1.7, -1.7) -- cycle; 
\draw[] (-0.5,-0.5,-0.5) -- (-.9, -2.1, -2.1) -- (-1.7, -1.7, -1.7) -- cycle; 

\draw[draw=green!20, fill=green!20, fill opacity=.8] (-0.5,0,-0.5) -- (-2.1, 1.6, -2.1) -- (-2.1, 0, -2.1) -- cycle; 
\draw[] (-0.5,0,-0.5) -- (-2.1, 1.6, -2.1) -- (-2.1, 0, -2.1) -- cycle; 
\draw[draw=green!20, fill=green!20, fill opacity=.8] (-0.5,0,-0.5) -- (-2.1, 0, -2.1) -- (-2.1, -.4, -2.1) -- cycle; 
\draw[] (-0.5,0,-0.5) -- (-2.1, 0, -2.1) -- (-2.1, -.4, -2.1) -- cycle;

\draw[draw=green!20, fill=green!20, fill opacity=.8] (0,-0.5,-0.5) -- (1.6, -2.1, -2.1) -- (0, -2.1, -2.1) -- cycle; 
\draw[] (0,-0.5,-0.5) -- (1.6, -2.1, -2.1) -- (0, -2.1, -2.1) -- cycle; 
\draw[draw=green!20, fill=green!20, fill opacity=.8] (0,-0.5,-0.5) -- (0, -2.1, -2.1) -- (-.4, -2.1, -2.1) --  cycle; 
\draw[] (0,-0.5,-0.5) -- (0, -2.1, -2.1) -- (-.4, -2.1, -2.1) --  cycle; 
\end{tikzpicture}
\end{center}

Proposition~\ref{prop:pyramidvolume} follows from the following lemma relating the volume function $\Vol_\sigma$ on $N_{\sigma,\R}$ to the volume function $\Vol_{\sigma^\rho}$ on the hyperplane $N_{\sigma^\rho,\R}\subseteq N_{\sigma,\R}$.

\begin{lemma}\label{lemma:pyramidvolume}
Under the hypotheses of Proposition~\ref{prop:pyramidvolume}, let $\sigma\in\Sigma(d)$ and $\rho\in\sigma(1)$. For any polytope $P\subseteq N_{\sigma^\rho,\R}$ and $a\in\R_{\geq 0}$, we have
\[
\Vol_\sigma\big(\conv(0,P+au_\rho )\big)=a(u_\rho*u_\rho)\cdot\Vol_{\sigma^\rho}(P).
\]
\end{lemma}

For intuition, we note that the polytope $\conv(0,P+au_\rho)$ appearing in the left-hand side of Lemma~\ref{lemma:pyramidvolume} is obtained from the polytope $P$ by first translating $P$ along the ray $\rho$, which is orthogonal to $N_{\sigma^\rho,\R}$, then taking the convex hull with the origin, the result of which can be thought of as a pyramid with $P$ as base and the origin as apex. The right-hand side can then be thought of as a ``base-times-height'' formula for the volume of this pyramid, where the ``height'' of the vector $au_\rho$ is $a(u_\rho*u_\rho)$. We now make this informal discussion precise.


\begin{proof}[Proof of Lemma~\ref{lemma:pyramidvolume}]
Let $\{v_\eta\mid \eta\in\sigma(1)\}\subseteq M_\sigma$ be the dual basis of $\{u_\eta\mid\eta\in\sigma(1)\}\subseteq N_\sigma$, defined uniquely by the equations
\[
v_\eta*u_\mu=\begin{cases}
1 & \mu=\eta\\
0 & \mu\neq\eta.
\end{cases}
\]
Recall that each ray generator of $\sigma^\rho$ is of the form $\pr^\rho(u_\eta)$ for a unique $\eta\in\sigma(1)\setminus\{\rho\}$; we claim that the dual vector of $\pr^\rho(u_\eta)$ in $M_{\sigma^\rho}$ is $v_\eta$---in other words, the dual vector of $\pr^\rho(u_\eta)$ is the same as the dual vector of $u_\eta$. To verify this, note that, for any $\eta,\mu\in\sigma(1)\setminus\{\rho\}$, we have
\begin{align*}
\pr^\rho(u_\eta)\,*^\rho\, v_\mu&=(u_\eta-\pr_\rho(u_\eta))*v_\mu\\
&=u_\eta*v_\mu\\
&=\begin{cases}
1 & \mu=\eta\\
0 & \mu\neq\eta,
\end{cases}
\end{align*}
where the first equality uses the decomposition of $u_\eta$ into its orthogonal components, along with the fact that $*^\rho$ is just the restriction of $*$, and the second equality uses that $\pr_\rho(u_\eta)$ is a multiple of $u_\rho$, along with $u_\rho*v_\mu=0$. 

Using these dual bases, we defined simplices in each of vector spaces $N_{\sigma,\R}$ and $N_{\sigma^\rho,\R}$ by
\[
\Delta(\sigma)=\conv(0,\{v_\eta\mid\eta\in\sigma(1)\})\subseteq N_{\sigma,\R}
\]
and
\[
\Delta(\sigma^\rho)=\conv(0,\{v_\eta\mid\eta\in\sigma(1)\setminus\{\rho\}\})\subseteq N_{\sigma^\rho,\R}.
\]
By our convention on how volumes are normalized in $N_{\sigma,\R}$ and $N_{\sigma^\rho,\R}$, along with our verification above that $\{v_\eta\mid\eta\in\sigma(1)\setminus\{\rho\}\}$ is the dual basis of the ray generators of $\sigma^\rho$, these simplices have unit volume:
\[
\Vol_\sigma(\Delta(\sigma))=\Vol_{\sigma^\rho}(\Delta(\sigma^\rho))=1.
\]
Notice that $\Delta(\sigma^\rho)$ is a facet of $\Delta(\sigma)$ and we can write $\Delta(\sigma)=\conv(v_\rho,\Delta(\sigma^\rho))$. If we project the vertex $v_\rho$ of $\Delta(\sigma)$ onto the line spanned by $\rho$, we obtain a new simplex
\[
\Delta_1(\sigma)=\conv(\pr_\rho(v_\rho),\Delta(\sigma^\rho)).
\]
Since the projection $\pr_\rho$ is parallel to the facet $\Delta(\sigma^\rho)$, it follows that
\[
\Vol_\sigma(\Delta_1(\sigma))=\Vol_\sigma(\Delta(\sigma))=\Vol_{\sigma^\rho}(\Delta(\sigma^\rho)).
\]
Now define a new simplex by sliding the vertex $\pr_\rho(v_\rho)$ along $\rho$ to the new vertex $au_\rho$:
\[
\Delta_2(\sigma)=\conv(a u_\rho,\Delta(\sigma^\rho)).
\]
By the standard projection formula, we have $\pr_\rho(v_\rho)=\frac{u_\rho}{u_\rho*u_\rho}$, from which we see that $\Delta_2(\sigma)$ is obtained from $\Delta_1(\sigma)$ by scaling the height of the vertex $\pr_\rho(v_\rho)$ by a factor of $a(u_\rho*u_\rho)$. It follows that the volume also scales by $a(u_\rho*u_\rho)$:
\[
\Vol_\sigma(\Delta_2(\sigma))=a(u_\rho*u_\rho)\cdot\Vol_\sigma(\Delta_1(\sigma))=a(u_\rho*u_\rho)\cdot\Vol_{\sigma^\rho}(\Delta(\sigma^\rho)).
\]
More concisely, we have proved that
\begin{equation}\label{eq:pyramidy}
\Vol_\sigma\big(\conv(a u_\rho,P)\big)=a(u_\rho*u_\rho)\cdot\Vol_{\sigma^\rho}(P)
\end{equation}
when $P=\Delta(\sigma^\rho)$. 

As a visual aid, we have depicted below the sequence of polytopes from the above discussion in the specific setting of a two-dimensional cone $\sigma$, which we have visualized in $\R^2$ with the usual dot product.

\begin{center}
\begin{tikzpicture}[scale=1.5]
\draw[draw=blue!10, fill=blue!10, fill opacity=.8] (0,0) -- (1.9,0) -- (1.9,1.9) -- cycle;
\node at (1,0) {$\bullet$};
\node at (1,1) {$\bullet$};
\node at (0,0) {$\bullet$};
\node[left] at (0,-.1) {$0$};
\node at (1.5,.65) {$\sigma$};
\draw[->, thick, black] (0,0) -- (2,0);
\draw[->, thick, black] (0,0) -- (2,2);
\node[right] at (2,0) {$\eta$};
\node[right] at (2,2) {$\rho$};
\node[below] at (1,0) {$u_\eta$};
\node[above] at (.9,1) {$u_\rho$};
\draw[thick,black] (-1.2,1.2) -- (1.2,-1.2);
\node[above] at (-1.2,1.2) {$N_{\sigma^\rho,\R}$};
\node at (0,1) {$\bullet$};
\node[above] at (0,1) {$v_\rho$};
\node at (1,-1) {$\bullet$};
\node[right] at (1,-1) {$v_\eta$};
\end{tikzpicture}
\hspace{50bp}
\begin{tikzpicture}[scale=1.5]
\draw[thick,black, fill=green!20, fill opacity=.8] (0,0) -- (1,-1) -- (0,1) -- cycle;
\draw[->] (0,0) -- (2,2);
\node at (0,0) {$\bullet$};
\node[left] at (0,-.1) {$0$};
\node[right] at (2,2) {$\rho$};
\draw[thick,black] (-1.2,1.2) -- (1.2,-1.2);
\node[above] at (-1.2,1.2) {$N_{\sigma^\rho,\R}$};
\node at (0,1) {$\bullet$};
\node[above] at (0,1) {$v_\rho$};
\node at (1,-1) {$\bullet$};
\node[right] at (1,-1) {$v_\eta$};
\draw[line width=2bp] (0,0) -- (1,-1);
\node at (1.1,0) {$\Delta(\sigma)$};
\draw[thick,->] (.8,.05) to [bend right=20] (.25,0);
\node at (.3,-1.1) {$\Delta(\sigma^\rho)$};
\draw[thick,->] (.2,-.9) to [bend left=20] (.47,-.53);
\end{tikzpicture}
\begin{tikzpicture}[scale=1.5]
\draw[thick,black, fill=green!20, fill opacity=.8] (0,0) -- (1,-1) -- (.5,.5) -- cycle;
\draw[->] (0,0) -- (2,2);
\node at (0,0) {$\bullet$};
\node[left] at (0,-.1) {$0$};
\node[right] at (2,2) {$\rho$};
\draw[thick,black] (-1.2,1.2) -- (1.2,-1.2);
\node[above] at (-1.2,1.2) {$N_{\sigma^\rho,\R}$};
\node at (.5,.5) {$\bullet$};
\node[above] at (.2,.5) {$\frac{u_\rho}{u_\rho*u_\rho}$};
\node at (1,-1) {$\bullet$};
\node[right] at (1,-1) {$v_\eta$};
\draw[line width=2bp] (0,0) -- (1,-1);
\node at (1.3,0) {$\Delta_1(\sigma)$};
\draw[thick,->] (.9,.05) to [bend right=20] (.35,0);
\node at (.3,-1.1) {$\Delta(\sigma^\rho)$};
\draw[thick,->] (.2,-.9) to [bend left=20] (.47,-.53);
\end{tikzpicture}
\hspace{50bp}
\begin{tikzpicture}[scale=1.5]
\draw[thick,black, fill=green!20, fill opacity=.8] (0,0) -- (1,-1) -- (1.7,1.7) -- cycle;
\draw[->] (0,0) -- (2,2);
\node at (0,0) {$\bullet$};
\node[left] at (0,-.1) {$0$};
\node[right] at (2,2) {$\rho$};
\draw[thick,black] (-1.2,1.2) -- (1.2,-1.2);
\node[above] at (-1.2,1.2) {$N_{\sigma^\rho,\R}$};
\node at (1.7,1.7) {$\bullet$};
\node[right] at (1.7,1.5) {$au_\rho$};
\node at (1,-1) {$\bullet$};
\node[right] at (1,-1) {$v_\eta$};
\draw[line width=2bp] (0,0) -- (1,-1);
\node at (.8,.1) {$\Delta_2(\sigma)$};
\node at (.3,-1.1) {$\Delta(\sigma^\rho)$};
\draw[thick,->] (.2,-.9) to [bend left=20] (.47,-.53);
\end{tikzpicture}
\end{center}

We now extend \eqref{eq:pyramidy} to any simplex $P\subseteq N_{\sigma^\rho,\R}$. To do so, first note that a simplex $P$ can be obtained from the specific simplex $\Delta(\sigma^\rho)$ by a composition of a translation and a linear transformation on $N_{\sigma^\rho,\R}$. Translating $P$ within $N_{\sigma^\rho,\R}$ does not affect the volume on either side of \eqref{eq:pyramidy}. Given a linear transformation $T$, on the other hand, we can extend it to a linear transformation $\widehat T$ on $N_{\sigma,\R}$ by simply fixing the vector $u_\rho$, in which case we have
\[
\widehat T(\conv(a u_\rho,P))=\conv(a u_\rho,T(P)).
\]
Since $\det(\widehat T)=\det(T)$ and linear transformations scale volumes by the absolute values of their determinants, we conclude that the equality in \eqref{eq:pyramidy} is preserved upon taking linear transforms of $P$:
\begin{align*}
\Vol_\sigma\big(\conv(a u_\rho,T(P))\big)&=\Vol_\sigma\big(\widehat T(\conv(a u_\rho,P))\big)\\
&=|\det(\widehat T)|\Vol_\sigma\big(\conv(a u_\rho,P)\big)\\
&=|\det(T)|\cdot a(u_\rho*u_\rho)\cdot\Vol_{\sigma^\rho}(P)\\
&=a(u_\rho*u_\rho)\cdot\Vol_{\sigma^\rho}(T(P)).
\end{align*}

Knowing that \eqref{eq:pyramidy} holds for simplices, we extend it to arbitrary polytopes $P\subseteq N_{\sigma^\rho,\R}$ by triangulating $P$ and applying \eqref{eq:pyramidy} to each simplex in the triangulation. The lemma then follows from \eqref{eq:pyramidy} along with the observation that $\conv(a u_\rho,P)$ is just a reflection of $\conv(0,P+au_\rho)$, so has the same volume.
\end{proof}

We now use Lemma~\ref{lemma:pyramidvolume} to prove Proposition~\ref{prop:pyramidvolume}.

\begin{proof}[Proof of Proposition~\ref{prop:pyramidvolume}]
For each top-dimensional cone $\sigma\in\Sigma(d)$ and $\rho\in\sigma(1)$, consider the polytope face $\F_\rho(P_{\sigma,*}(z))\subseteq P_{\sigma,*}(z)$. By definition, we have
\[
\F_\rho(P_{\sigma,*}(z))=\F^\rho(P_{\sigma,*}(z))+w_{\rho,*}(z).
\]
Noting that $w_{\rho,*}(z)=\frac{z_\rho}{u_\rho*u_\rho}u_\rho$, Lemma~\ref{lemma:pyramidvolume} computes the volume of the pyramid $\conv(0,\F_\rho(P_{\sigma,*}(z)))$:
\begin{equation}\label{eq:pyramid!!}
\Vol_\sigma\big(\conv(0,\F_\rho(P_{\sigma,*}(z)))\big)=z_\rho\Vol_{\sigma^\rho}(\F^\rho(P_{\sigma,*}(z))=z_\rho\Vol_{\sigma^\rho}(P_{\sigma^\rho,*^\rho}(z^\rho)),
\end{equation}
where the second equality is an application of \eqref{eq:polytranslate}.

Next, note that we can decompose each polytope $P_{\sigma,*}(z)$ into pyramids over the faces $\F_\rho(P_{\sigma,*}(z))$ with $\rho\in\sigma(1)$, implying that
\begin{equation}\label{eq:pyramid!}
\Vol_\sigma(P_{\sigma,*}(z))=\sum_{\rho\in\sigma(1)}\Vol_\sigma\big(\conv(0,\F_\rho(P_{\sigma,*}(z))\big).
\end{equation}
We then compute:
\begin{align*}
\Vol_{\Sigma,\omega,*}(z)&=\sum_{\sigma\in\Sigma(d)}\omega(\sigma)\Vol_\sigma(P_{\sigma,*}(z))\\
&=\sum_{\sigma\in\Sigma(d)}\omega(\sigma)\sum_{\rho\in\sigma(1)}\Vol_\sigma\big(\conv(0,\F_\rho(P_{\sigma,*}(z))\big)\\
&=\sum_{\sigma\in\Sigma(d)}\omega(\sigma)\sum_{\rho\in\sigma(1)}z_\rho\Vol_{\sigma^\rho}(P_{\sigma^\rho,*^\rho}(z^\rho))\\
&=\sum_{\rho\in\Sigma(1)}z_\rho\sum_{\sigma^\rho\in\Sigma^\rho(d-1)}\omega^\rho(\sigma^\rho)\Vol_{\sigma^\rho}(P_{\sigma^\rho,*^\rho}(z^\rho))\\
&=\sum_{\rho\in\Sigma(1)}z_\rho\Vol_{\Sigma^\rho,\omega^\rho,*^\rho}(z^\rho),
\end{align*}
where the first equality is the definition of $\Vol_{\Sigma,\omega,*}(z)$, the second and third are \eqref{eq:pyramid!} and \eqref{eq:pyramid!!}, respectively, the fourth follows from the definition of $\omega^\rho$ and the fact that cones in $\Sigma^\rho(d-1)$ are in bijection with the cones in $\Sigma(d)$ containing $\rho$ via $\sigma^\rho\leftrightarrow\sigma$, and the fifth is the definition of $\Vol_{\Sigma^\rho,\omega^\rho,*^\rho}(z^\rho)$.
\end{proof}

\subsection{Mixed volumes and facets}

The aim of this subsection is to enhance Proposition~\ref{prop:pyramidvolume} to the following more general statement about mixed volumes. See \cite[Lemma 5.1.5]{Schneider} for the analogous result in the classical setting of strongly isomorphic polytopes.

\begin{proposition}~\label{prop:pyramidmixed}
Let $\Sigma\subseteq N_\R$ be a simplicial $d$-fan with weight function $\omega:\Sigma(d)\rightarrow\R_{>0}$,  let $*\in\Inn(N_\R)$ be an inner product, and let $z_1,\dots,z_d\in\oCub(\Sigma,*)$ be pseudocubical values. Then
\[
\MVol_{\Sigma,\omega,*}(z_1,\dots,z_d)=\sum_{\rho\in\Sigma(1)}z_{1,\rho}\MVol_{\Sigma^\rho,\omega^\rho,*^\rho}(z_2^\rho,\dots,z_d^\rho).
\]
\end{proposition}

\begin{proof}
We proceed by induction on $d$. If $d=1$, then mixed volumes are just volumes, in which case Proposition~\ref{prop:pyramidmixed} is a special case of Proposition~\ref{prop:pyramidvolume}. Assume, now, that Proposition~\ref{prop:pyramidmixed} holds in dimension less than $d>1$. Define
\[
F(z_1,\dots,z_d)=\sum_{\rho\in\Sigma(1)}z_{1,\rho}\MVol_{\Sigma^\rho,\omega^\rho,*^\rho}(z_2^\rho,\dots,z_d^\rho).
\]
To prove that $F=\MVol_{\Sigma,*,\omega}$, Proposition~\ref{prop:mvolchar} tells us that it suffices to prove that $F$ is (1) symmetric, (2) multilinear, and (3) normalized correctly with respect to volume; we check these properties in reverse order.

To check (3), we note that
\begin{align*}
F(z,\dots,z)&=\sum_{\rho\in\Sigma(1)}z_{\rho}\MVol_{\Sigma^\rho,\omega^\rho,*^\rho}(z^\rho,\dots,z^\rho)\\
&=\sum_{\rho\in\Sigma(1)}z_{\rho}\Vol_{\Sigma^\rho,\omega^\rho,*^\rho}(z^\rho)\\
&=\Vol_{\Sigma,\omega,*}(z),
\end{align*}
where the first equality is the definition of $F$, the second is Proposition~\ref{prop:mvolchar} Part (3), and the third is Proposition~\ref{prop:pyramidvolume}.

To check (2), there are two cases to consider: linearity in the first coordinate and linearity in every other coordinate. Linearity in the first coordinate follows quickly from the definition of $F$, while linearity in every other coordinate follows from Proposition~\ref{prop:mvolchar} Part (2) applied to $(\Sigma^\rho,*^\rho,\omega^\rho)$.

Finally, to check (1), we first note that Proposition~\ref{prop:mvolchar} Part (1) applied to $(\Sigma^\rho,*^\rho,\omega^\rho)$ implies that $F$ is symmetric in the entries $z_2,\dots,z_d$. Thus, it remains to prove that $F$ is invariant under transposing $z_1$ and $z_2$. To do so, we first apply the induction hypothesis to the mixed volumes appearing in the definition of $F$ to obtain
\begin{equation}\label{eq:mixedinduction}
F(z_1,\dots,z_d)=\sum_{\rho\in\Sigma(1)}z_{1,\rho}\sum_{\eta^\rho\in\Sigma^\rho(1)}z_{2,\eta^\rho}^\rho\MVol_{\Sigma^{\rho,\eta},\omega^{\rho,\eta},*^{\rho,\eta}}(z_3^{\rho,\eta},\dots,z_d^{\rho,\eta}),
\end{equation}
where, to avoid the proliferation of parentheses and superscripts, we have written, for example, $\Sigma^{\rho,\eta}$ as short-hand for $(\Sigma^\rho)^{\eta^\rho}$. Notice that the mixed volumes appearing in the right-hand side of \eqref{eq:mixedinduction} are mixed volumes associated to faces of faces. Proposition~\ref{prop:faceofaface} tells us that the $\eta^\rho$-face of the $\rho$-face of a normal complex is the same as the $\tau$ face of the original normal complex, where $\tau\in\Sigma(2)$ is the $2$-cone containing $\rho$ and $\eta$ as rays. Therefore,
\[
\MVol_{\Sigma^{\rho,\eta},\omega^{\rho,\eta},*^{\rho,\eta}}(z_3^{\rho,\eta},\dots,z_d^{\rho,\eta})=\MVol_{\Sigma^\tau,\omega^\tau,*^\tau}(z_3^\tau,\dots,z_d^\tau).
\]
 Keeping in mind that each $2$-cone $\tau$ appears twice in \eqref{eq:mixedinduction}, once for each ordering of the rays, we have
\[
F(z_1,\dots,z_d)=\sum_{\tau\in\Sigma(2)\atop \tau(1)=\{\rho,\eta\}}(z_{1,\rho}z_{2,\eta^\rho}^\rho+z_{1,\eta}z_{2,\rho^\eta}^\eta)\MVol_{\Sigma^\tau,\omega^\tau,*^\tau}(z_3^\tau,\dots,z_d^\tau).
\]
Therefore, it remains to prove that $z_{1,\rho}z_{2,\eta^\rho}^\rho+z_{1,\eta}z_{2,\rho^\eta}^\eta$ is invariant under transposing $1$ and $2$. Computing directly from the definition of $z^\rho$, we have
\[
z_{1,\rho}z_{2,\eta^\rho}^\rho+z_{1,\eta}z_{2,\rho^\eta}^\eta=z_{1,\rho}\big(z_{2,\eta}-w_{\rho,*}(z_2)*u_\eta\big)+z_{1,\eta}\big(z_{2,\rho}-w_{\eta,*}(z_2)*u_\rho\big),
\]
from which we see that it suffices to prove that both
\[
z_{1,\rho}w_{\rho,*}(z_2)*u_\eta\;\;\;\text{ and }\;\;\;z_{1,\eta}w_{\eta,*}(z_2)*u_\rho
\]
are invariant under transposing $1$ and $2$. This invariance follows from the computations
\[
w_{\rho,*}(z_2)=\frac{z_{2,\rho}}{u_\rho*u_\rho}u_\rho\;\;\;\text{ and }\;\;\;w_{\eta,*}(z_2)=\frac{z_{2,\eta}}{u_\eta*u_\eta}u_\eta.\qedhere
\]
\end{proof}

The following analytic consequence of Proposition~\ref{prop:pyramidmixed} will be useful in our computations in the next section.

\begin{corollary}\label{cor:derivatives}
In addition to the hypotheses of Proposition~\ref{prop:pyramidmixed}, assume that $\Cub(\Sigma,*)$ is nonempty. Then for any fixed $z_1,\dots,z_k\in\Cub(\Sigma,*)$, we have
\[
\frac{\partial}{\partial z_\rho}\MVol_{\Sigma,\omega,*}(z_1,\dots,z_k,\underbrace{z,\dots,z}_{d-k})=(d-k)\MVol_{\Sigma^\rho,\omega^\rho,*^\rho}(z_1^\rho,\dots,z_k^\rho,\underbrace{z^\rho,\dots,z^\rho}_{d-k-1}).
\]
\end{corollary}

\begin{proof}
The assumption that $\Cub(\Sigma,*)\neq\emptyset$ implies that $\MVol_{\Sigma,\omega,*}(z_1,\dots,z_k,z,\dots,z)$ is a degree $d-k$ polynomial in $\R[z_\rho\mid\rho\in\Sigma(1)]$, so the derivatives are well-defined. Proposition~\ref{prop:pyramidmixed} and symmetry of mixed volumes imply that
\[
\frac{\partial}{\partial z_{i,\rho}}\MVol_{\Sigma,\omega,*}(z_1,\dots,z_d)=\MVol_{\Sigma^\rho,\omega^\rho,*^\rho}(z_1^\rho,\dots,z_{i-1}^\rho,z_{i+1}^\rho,\dots,z_d^\rho).
\]
Viewing $\MVol_{\Sigma,\omega,*}(z_1,\dots,z_k,z,\dots,z)$ as the composition of $\MVol_{\Sigma,\omega,*}(z_1,\dots,z_d)$ with the specialization
\[
z_{k+1}=\cdots=z_d=z, 
\]
the result then follows from the multivariable chain rule.
\end{proof}

\section{Alexandrov--Fenchel inequalities}

One of the most consequential properties of mixed volumes of polytopes (or, more generally, of mixed volumes of convex bodies) is the \textbf{Alexandrov--Fenchel inequalities}. Given polytopes $P_1,\dots,P_d$ in a $d$-dimensional real vector space $V$ with volume function $\Vol$, the Alexandrov--Fenchel inequalities state that
\[
\MVol(P_1,P_2,P_3,\dots,P_d)^2\geq \MVol(P_1,P_1,P_3,\dots,P_d)\MVol(P_2,P_2,P_3,\dots,P_d)
\] 
 (see, for example, \cite[Theorem~7.3.1]{Schneider} for a proof and historical references). It is our aim in this section to study Alexandrov--Fenchel inequalities in the setting of mixed volumes of normal complexes.

Let $\Sigma\subseteq N_\R$ be a simplicial $d$-fan, $\omega:\Sigma(d)\rightarrow\R_{>0}$ a weight function, and $*\in\Inn(N_\R)$ an inner product. We say that the triple $(\Sigma,\omega,*)$ is \textbf{Alexandrov--Fenchel}, or just \textbf{AF} for short, if $\Cub(\Sigma,*)\neq\emptyset$ and 
\[
\MVol_{\Sigma,\omega,*}(z_1,z_2,z_3,\dots,z_d)^2\geq\MVol_{\Sigma,\omega,*}(z_1,z_1,z_3,\dots,z_d)\MVol_{\Sigma,\omega,*}(z_2,z_2,z_3,\dots,z_d)
\]
for all $z_1,\dots,z_d\in\Cub(\Sigma,*)$. In this section, we prove the following result, which provides sufficient conditions for proving that a triple $(\Sigma,\omega,*)$ is AF.

\begin{theorem}\label{thm:reduce}
Let $\Sigma\subseteq N_\R$ be a simplicial $d$-fan, $\omega:\Sigma(d)\rightarrow\R_{>0}$ a weight function, and $*\in\Inn(N_\R)$ an inner product such that $\Cub(\Sigma,*)\neq\emptyset$. The triple $(\Sigma,\omega,*)$ is AF if the following two conditions are satisfied:
\begin{enumerate}
\item[(i)] $\Sigma^\tau\setminus\{0\}$ is connected for any cone $\tau\in\Sigma(k)$ with $k\leq d-3$;
\item[(ii)] $\mathrm{Hess}\big(\Vol_{\Sigma^\tau,\omega^\tau,*^\tau}(z)\big)$ has exactly one positive eigenvalue for any $\tau\in\Sigma(d-2)$.
\end{enumerate}
\end{theorem}

\begin{remark}
Condition (i) in Theorem~\ref{thm:reduce} can be thought of as requiring that the fan $\Sigma$ does not have any ``pinch'' points. For example, in dimension four, this condition rules out fans that locally look like a pair of four-dimensional cones meeting along a ray, because the star fan associated to that ray would comprise two three-dimensional cones that meet only at the origin.
\end{remark}

\begin{remark}
Condition (ii) of Theorem~\ref{thm:reduce} concerns only the two-dimensional stars of $\Sigma$. Since the volume polynomial of a two-dimensional fan is a quadratic form, the Hessians appearing in Condition (ii) are constant matrices. Condition (ii) can be viewed as an analogue of the Brunn--Minkowski inequality for polygons. For an example of a two-dimensional (tropical) fan that does not satisfy Condition (ii), see \cite{BabaeeHuh}.
\end{remark}

\subsection{Proof of Theorem~\ref{thm:reduce}}

Our proof of Theorem~\ref{thm:reduce} is largely inspired by a proof of the classical Alexandrov--Fenchel inequalities recently developed by Cordero-Erausquin, Klartag, Merigot, and Santambrogio \cite{OneMoreProof}---for which the key geometric input is Proposition~\ref{prop:pyramidmixed}. While the arguments in \cite{OneMoreProof} can be employed in this setting more-or-less verbatim, we present a more streamlined proof using ideas regarding Lorentzian polynomials recently developed by Br\"and\'en and Leake \cite{BrandenLeake}. Before presenting a proof of Theorem~\ref{thm:reduce}, we pause to introduce key ideas regarding Lorentzian polynomials.

\subsubsection{Lorentzian polynomials on cones}

One way to view the AF inequalities is as the nonpositivity of the $2\times 2$ matrix
\[
\left[
\begin{array}{cc}
\MVol_{\Sigma,\omega,*}(z_1,z_1,z_3,\dots,z_d) & \MVol_{\Sigma,\omega,*}(z_1,z_2,z_3,\dots,z_d)\\
\MVol_{\Sigma,\omega,*}(z_2,z_1,z_3,\dots,z_d) & \MVol_{\Sigma,\omega,*}(z_2,z_2,z_3,\dots,z_d)
\end{array}
\right],
\] 
and this nonpositivity is equivalent to the matrix having exactly one positive eigenvalue. Lorentzian polynomials are a clever tool for capturing the essence of this observation, and are therefore a natural setting for understanding AF-type inequalities. 

Our discussion of Lorentzian polynomials follows Br\"and\'en and Leake \cite{BrandenLeake}. Suppose that $C\subseteq\R_{>0}^n$ is a nonempty open convex cone, and let $f\in\R[x_1,\dots,x_n]$ be a homogeneous polynomial of degree $d$. For each $i=1,\dots,n$ and $v=(v_1,\dots,v_n)\in\R^n$, we use the following shorthand for partial  and directional derivatives
\[
\partial_i=\frac{\partial}{\partial x_i}\;\;\;\text{ and }\;\;\;\partial_v=\sum_{i=1}^n v_i\partial_i.
\]
We say that $f$ is \textbf{$C$-Lorentzian} if, for all $v_1,\dots,v_d\in C$, 
\begin{enumerate}
\item[(P)] $\partial_{v_1}\cdots\partial_{v_d}f>0$, and
\item[(H)] $\Hess(\partial_{v_3}\cdots\partial_{v_d}f)$ has exactly one positive eigenvalue.
\end{enumerate}

To relate Lorentzian polynomials back to AF-type inequalities, we recall the following key observation (see \cite[Proposition~4.4]{BrandenHuh}).

\begin{lemma}\label{lem:lorentzian}
Let $C\subseteq\R_{>0}^n$ be a nonempty open convex cone, and let $f\in\R[x_1,\dots,x_n]$ be $C$-Lorentzian. Then for all $v_1,v_2,v_3\dots,v_d\in C$, we have
\[
\big(\partial_{v_1}\partial_{v_2}\partial_{v_3}\cdots\partial_{v_d}f\big)^2\geq \big(\partial_{v_1}\partial_{v_1}\partial_{v_3}\cdots\partial_{v_d}f\big) \big(\partial_{v_2}\partial_{v_2}\partial_{v_3}\cdots\partial_{v_d}f\big).
\]
\end{lemma}

\begin{proof}
Consider the symmetric $2\times 2$ matrix
\[
M=\left[
\begin{array}{cc}
\partial_{v_1}\partial_{v_1}\partial_{v_3}\cdots\partial_{v_d}f & \partial_{v_1}\partial_{v_2}\partial_{v_3}\cdots\partial_{v_d}f\\
\partial_{v_2}\partial_{v_1}\partial_{v_3}\cdots\partial_{v_d}f & \partial_{v_2}\partial_{v_2}\partial_{v_3}\cdots\partial_{v_d}f
\end{array}
\right].
\]
By (P), the entries of $M$ are positive, so the Peron--Frobenius Theorem implies that $M$ has \emph{at least} one positive eigenvalue. On the other hand, $M$ is a principal minor of $\Hess(\partial_{v_3}\cdots\partial_{v_d}f)$, which, by (H), has exactly one positive eigenvalue; thus, it follows from Cauchy's Interlacing Theorem that $M$ has \emph{at most} one positive eigenvalue. Therefore $M$ has exactly one positive eigenvalue, implying that the determinant of $M$ is nonpositive, proving the lemma.
\end{proof}

The following result, proved by Br\"and\'en and Leake \cite{BrandenLeake}, is particularly useful for the study of Lorentzian polynomials on cones. We view this result as an effective implementation of the key insights in \cite{OneMoreProof}; in essence, it eliminates the need for one of the induction parameters in \cite{OneMoreProof} because that induction parameter is captured within the recursive nature of Lorentzian polynomials.

\begin{lemma}[\cite{BrandenLeake}, Proposition 2.4]\label{lem:lorentziancheck}
Let $C\subseteq\R_{>0}^n$ be a nonempty open convex cone, and let $f\in\R[x_1,\dots,x_n]$ be a homogeneous polynomial of degree $d$. If
\begin{enumerate}
\item $\partial_{v_1}\cdots\partial_{v_d}f>0$ for all $v_1,\dots,v_d\in C$, 
\item $\Hess\big(\partial_{v_1}\cdots\partial_{v_{d-2}}f\big)$ is irreducible\footnote{An $n\times n$ matrix $M$ is \textbf{irreducible} if the associated adjacency graph---the undirected graph on $n$ labeled vertices with an edge between the $i$th and $j$th vertex whenever the $(i,j)$ entry of $M$ is nonzero---is connected.} and has nonnegative off-diagonal entries for all $v_1,\dots,v_{d-2}\in C$, and
\item $\partial_i f$ is $C$-Lorentzian for all $i=1,\dots,n$,
\end{enumerate}
then $f$ is $C$-Lorentzian.
\end{lemma}

\subsubsection{Lorentzian volume polynomials}

We now discuss how the above discussion of Lorentzian polynomials on cones can be used to study mixed volumes of normal complexes. Let $\Sigma\subseteq N_\R$ be a simplicial $d$-fan, $\omega:\Sigma(d)\rightarrow\R_{>0}$ a weight function, and $*\in\Inn(N_\R)$ an inner product. We assume that $\Cub(\Sigma,*)\neq\emptyset$, in which case the function $\Vol_{\Sigma,\omega,*}:\Cub(\Sigma,*)\rightarrow\R$ is a homogeneous polynomial of degree $d$ in $\R[z_\rho\mid\rho\in\Sigma(1)]$. By Proposition~\ref{prop:mvolchar}(3), we have
\[
\Vol_{\Sigma,\omega,*}(z)=\MVol_{\Sigma,\omega,*}(z,\dots,z).
\]
It then follows from Proposition~\ref{prop:mvolchar}(1) and (2) (and the chain rule) that
\begin{equation}\label{eq:partials}
\partial_{z_1}\cdots\partial_{z_k}\Vol_{\Sigma,\omega,*}(z)=\frac{d!}{d-k!}\MVol_{\Sigma,\omega,*}(z_1,\dots,z_k,\underbrace{z,\dots,z}_{d-k})
\end{equation}
for any $z_1,\dots,z_k\in\Cub(\Sigma,*)$. In particular, in order to prove that $(\Sigma,\omega,*)$ is AF, we now see that it suffices (by Lemma~\ref{lem:lorentzian}) to prove that $\Vol_{\Sigma,\omega,*}$ is $\Cub(\Sigma,*)$-Lorentzian. Thus, Theorem~\ref{thm:reduce} is a consequence of the following stronger result.

\begin{theorem}\label{thm:lorentzian}
Let $\Sigma\subseteq N_\R$ be a simplicial $d$-fan, $\omega:\Sigma(d)\rightarrow\R_{>0}$ a weight function, and $*\in\Inn(N_\R)$ an inner product such that $\Cub(\Sigma,*)\neq\emptyset$. Then $\Vol_{\Sigma,\omega,*}$ is $\Cub(\Sigma,*)$-Lorentzian if the following two conditions are satisfied:
\begin{enumerate}
\item[(i)] $\Sigma^\tau\setminus\{0\}$ is connected for any cone $\tau\in\Sigma(k)$ with $k\leq d-3$;
\item[(ii)] $\mathrm{Hess}\big(\Vol_{\Sigma^\tau,\omega^\tau,*^\tau}(z)\big)$ has exactly one positive eigenvalue for any $\tau\in\Sigma(d-2)$.
\end{enumerate}
\end{theorem}

\begin{proof}
We prove Theorem~\ref{thm:lorentzian} by induction on $d$. 

First consider the base case $d=2$ (in which case Condition (i) is vacuous). Note that $\Vol_{\Sigma,\omega,*}$ satisfies (P) by \eqref{eq:partials} and the positivity of mixed volumes (Proposition~\ref{prop:positive}), while (H) for $\Vol_{\Sigma,\omega,*}$ is equivalent to Condition (ii). Therefore, Theorem~\ref{thm:lorentzian} holds when $d=2$. 

Now let $d>2$ and assume $(\Sigma,\omega,*)$ satisfies Conditions (i) and (ii) in Theorem~\ref{thm:lorentzian}. To prove that $\Vol_{\Sigma,\omega,*}$ is $\Cub(\Sigma,*)$-Lorentzian, we use Lemma~\ref{lem:lorentziancheck}. Translating the three conditions of Lemma~\ref{lem:lorentziancheck} using \eqref{eq:partials}, we must prove that
\begin{enumerate}
\item $\MVol_{\Sigma,\omega,*}(z_1,\dots,z_d)>0$ for all $z_1,\dots,z_d\in\Cub(\Sigma,*)$,
\item $\Hess\big(\MVol_{\Sigma,\omega,*}(z_1,\dots,z_{d-2},z,z)\big)$ is irreducible and has nonnegative off-diagonal entries for all $z_1,\dots,z_{d-2}\in \Cub(\Sigma,*)$, and
\item $\partial_\rho \Vol_{\Sigma,\omega,*}(z)$ is $\Cub(\Sigma,*)$-Lorentzian for all $\rho\in\Sigma(1)$.
\end{enumerate}

Note that (1) is just the positivity of mixed volumes (Proposition~\ref{prop:positive}). To prove (3), note that Proposition~\ref{prop:mvolchar}(3) and Corollary~\ref{cor:derivatives} (with $k=0$) together imply that
\[
\partial_\rho\Vol_{\Sigma,\omega,*}(z)=d\Vol_{\Sigma^\rho,\omega^\rho,*^\rho}(z^\rho).
\]
Applying the induction hypothesis to $(\Sigma^\rho,\omega^\rho,*^\rho)$---which we can do because any star fan of $\Sigma^\rho$ is a star fan of $\Sigma$, so our assumption that $(\Sigma,\omega,*)$ satisfies the two conditions of Theorem~\ref{thm:lorentzian} implies that $(\Sigma^\rho,\omega^\rho,*^\rho)$ also satisfies the two conditions of Theorem~\ref{thm:lorentzian}---implies that $\partial_\rho\Vol_{\Sigma,\omega,*}(z)$ is Lorentzian, verifying (3).

Finally, to prove (2), we use Corollary~\ref{cor:derivatives} to compute
\[
\partial_\rho\MVol_{\Sigma,\omega,*}(z_1,\dots,z_{d-2},z,z)=2\MVol_{\Sigma,\omega,*}(z_1^\rho,\dots,z_{d-2}^\rho,z^\rho).
\]
If $\tau\in\Sigma(2)$ with rays $\rho$ and $\eta$, then 
\[
z^\rho_{\eta^\rho}=z_\eta-w_{\rho,*}(z)*u_\eta=z_\eta-\frac{u_\rho*u_\eta}{u_\rho*u_\rho}z_\rho,
\]
from which it follows that,
\begin{equation}\label{eq:secondderivative1}
\partial_\eta\partial_\rho\MVol_{\Sigma,\omega,*}(z_1,\dots,z_{d-2},z,z)=2\MVol_{\Sigma^\tau,\omega^\tau,*^\tau}(z_1^\tau,\dots,z_{d-2}^\tau)
\end{equation}
On the other hand, if $\rho$ and $\eta$ do not lie on a common cone $\tau\in\Sigma(2)$, then
\begin{equation}\label{eq:secondderivative2}
\partial_\eta\partial_\rho\MVol_{\Sigma,\omega,*}(z_1,\dots,z_{d-2},z,z)=0.
\end{equation}
The positivity of mixed volumes for cubical values, along with \eqref{eq:secondderivative1} and \eqref{eq:secondderivative2}, then implies that $\Hess\big(\MVol_{\Sigma,\omega,*}(z_1,\dots,z_{d-2},z,z)\big)$ has nonnegative off-diagonal entries that are positive whenever the row and column index are the rays of a cone $\tau\in\Sigma(2)$. The first condition in Theorem~\ref{thm:lorentzian} implies that we can travel from any ray of $\Sigma$ to any other ray by passing only through the relative interiors of one- and two-dimensional cones, which then implies that $\Hess\big(\MVol_{\Sigma,\omega,*}(z_1,\dots,z_{d-2},z,z)\big)$ is irreducible, concluding the proof.
\end{proof}

\section{Application: the Heron--Rota--Welsh conjecture}

As an application of our developments regarding mixed volumes of normal complexes, we show in this section how Theorem~\ref{thm:reduce} can be used to prove the Heron--Rota--Welsh conjecture, which states that the coefficients of the characteristic polynomial of any matroid are log-concave. The bridge between matroids and mixed volumes is the Bergman fan; we begin this section by briefly recalling relevant notions regarding matroids and Bergman fans.

\subsection{Matroids and Bergman fans}

A \textbf{(loopless) matroid} $\sM=(E,\cL)$ consists of a finite set $E$, called the \textbf{ground set}, and a collection of subsets $\cL\subseteq 2^E$, called \textbf{flats}, which satisfy the following three conditions:
\begin{enumerate}
\item[(F1)] $\emptyset\in\cL$,
\item[(F2)] if $F_1,F_2\in\cL$, then $F_1\cap F_2\in\cL$, and
\item[(F3)] if $F\in\cL$, then every element of $E\setminus F$ is contained in exactly one flat that is minimal among the flats that strictly contain $F$.
\end{enumerate}

We do not give a comprehensive overview of matroids; rather, we settle for a brief introduction of key concepts. For a more complete treatment, see Oxley's book \cite{Oxley}.

The \textbf{closure} of a set $S\subseteq E$, denoted $\cl(S)$, is the smallest flat containing $S$. A set $I\subseteq E$ is called \textbf{independent} if $\cl(I_1)\subsetneq\cl(I_2)$ for any $I_1\subsetneq I_2\subseteq I$. The \textbf{rank} of a set $S\subseteq E$, denoted $\rk(S)$, is the maximum size of an independent subset of $S$, and the \textbf{rank of $\sM$}, denoted $\rk(\sM)$ is defined to be the rank of $E$.  While we have chosen to characterize matroids in terms of their flats, we note that matroids can also be characterized in terms of their independent sets or their rank function.

A \textbf{flag of flats (of length $k$) in $\sM$} is a chain of the form
\[
\cF=(F_1\subsetneq\cdots\subsetneq F_k)\;\;\;\text{ with }\;\;\;F_1,\dots, F_k\in\cL.
\]
It can be checked from the matroid axioms that every maximal flag has one flat of each rank $0,\dots,\rk(\sM)$. We let $\Delta_\sM$ denote the set of flags of flats, which naturally has the structure of a simplicial complex of dimension $\rk(\sM)+1$. Since every maximal flag contains $\emptyset$ and $E$, we often restrict our attention to studying proper flats. We use the notation $\cL^*=\cL\setminus\{\emptyset,E\}$ for the set  of proper flats and $\Delta_\sM^*$ for the set of flags of proper flats, which is a simplicial complex of dimension $\rk(\sM)-1$.

Given a matroid $\sM$, consider the vector space $\R^E$ with basis $\{v_e\mid e\in E\}$. For each subset $S\subseteq E$, define
\[
v_S=\sum_{e\in S}v_e\in\R^E.
\]
Set $N_\R=\R^E/\R v_E$ and denote the image of $v_S$ in the quotient space $N_\R$ by $u_S$. For each flag $\cF=(F_1\subsetneq\cdots\subsetneq F_k)\in\Delta_\sM^*$, define a polyhedral cone
\[
\sigma_\cF=\R_{\geq 0}\{u_{F_1},\dots,u_{F_k}\}\subseteq N_\R.
\]
The \textbf{Bergman fan of $\sM$}, denoted $\Sigma_\sM$, is the polyhedral fan
\[
\Sigma_\sM=\{\sigma_\cF\mid \cF\in\Delta^*_\sM\}.
\] 
Note that $\Sigma_\sM$ is simplicial, pure of dimension $d=\rk(\sM)-1$, and marked by the vectors $u_F$.

Consider a cone $\sigma_\cF\in\Sigma_\sM(d-1)$ corresponding to a flag
\[
\cF=(F_1\subsetneq\cdots\subsetneq F_{k-1}\subsetneq F_{k+1}\subsetneq\cdots\subsetneq F_d)\;\;\;\text{ with }\;\;\;\rk(F_i)=i.
\]
The $d$-cones containing $\sigma_\cF$ are indexed by flats $F$ with $F_{k-1}\subsetneq F\subsetneq F_{k+1}$. If there are $\ell$ such flats, then (F3) implies that 
\[
\sum_{F\in\cL \atop F_{k-1}\subsetneq F\subsetneq F_{k+1}}u_F=(\ell-1)u_{F_{k-1}}+u_{F_{k+1}}.
\]
Since the right-hand side lies in $N_{\sigma_{\cF},\R}$, this observation implies that $\Sigma_{\sM}$ is balanced (tropical with weights all equal to $1$).

In order to check that Bergman fans are AF, we require a working understanding of the star fans of Bergman fans. Consider a cone $\sigma_\cF$ associated to a flat $\cF=(F_1\subsetneq\cdots\subsetneq F_k)$. Set $F_0=\emptyset$ and $F_{k+1}=E$, and for each $j=0,\dots,k$ consider the \textbf{matroid minor} $\sM[F_j,F_{j+1}]$, which is the matroid on ground set $F_{j+1}\setminus F_j$ with flats of the form $F\setminus F_j$ where $F$ is a flat of $\sM$ satisfying $F_j\subseteq F\subseteq F_{j+1}$. Notice that the star fan $\Sigma_\sM^{\sigma_{\cF}}$ lives in the quotient space
\[
N_\R^{\sigma_\cF}=\frac{N_\R}{\R\{u_{F_1},\dots,u_{F_{k}}\}}=\frac{\R^E}{\R\{v_{F_1},\dots,v_{F_{k+1}}\}} = \bigoplus_{j=0}^k\frac{\R^{F_{k+1}\setminus F_k}}{\R v_{F_{k+1}\setminus F_k}},
\]
and one checks that this natural isomorphism of vector spaces identifies the star of $\Sigma_\sM$ at $\sigma_{\cF}$ as the product of the Bergman fans of the associated matroid minors:
\begin{equation}\label{eq:product}
\Sigma_\sM^{\sigma_\cF}=\prod_{j=0}^k\Sigma_{\sM[F_j,F_{j+1}]}.
\end{equation}

\subsection{Bergman fans are AF}

We are now ready to use Theorem~\ref{thm:reduce} to prove that Bergman fans of matroids are AF.

\begin{theorem}\label{thm:bergman}
Let $\sM$ be a matroid of rank $d+1$ and let $\Sigma_\sM\subseteq N_\R$ be the associated Bergman fan. If $*\in\Inn(N_\R)$ is any inner product with $\Cub(\Sigma_\sM,*)\neq\emptyset$, then $(\Sigma_\sM,*)$ is AF.
\end{theorem}

\begin{remark}
We are assuming the weight function $\omega$ is equal to $1$ because, as noted in the previous subsection, $\Sigma_\sM$ is balanced. Thus, we omit $\omega$ from the notation in this section.
\end{remark}

To prove Theorem~\ref{thm:bergman}, we verify the two conditions of Theorem~\ref{thm:reduce}. We accomplish this through the following three lemmas. The first lemma verifies that Bergman fans satisfy (a slight strengthening of) Condition (i) of Theorem~\ref{thm:bergman}.

\begin{lemma}\label{lem:cond1}
$\Sigma_\sM^{\sigma_\cF}\setminus\{0\}$ is connected for any cone $\sigma_\cF\in\Sigma_\sM(k)$ with $k\leq d-2$.
\end{lemma}

\begin{proof}
We begin by arguing that $\Sigma_\sM\setminus\{0\}$ is connected for any matroid of rank at least $3$. It suffices to prove that, for any two rays $\rho_F,\rho_{F'}\in\Sigma_\sM(1)$ associated to flats $F,F'\in\cL^*$, there are sequences $\rho_1,\dots,\rho_\ell\in\Sigma_\sM(1)$ and $\tau_1,\dots,\tau_{\ell+1}\in\Sigma_\sM(2)$ such that
\[
\rho_F\prec\tau_1\succ\rho_1\prec\cdots\succ\rho_\ell\prec\tau_{\ell+1}\succ\rho_{F'}.
\]
If $F\cap F'=G\neq\emptyset$, then $G\in\cL^*$ by (F2) and the following is such a sequence
\[
\rho_F\prec\tau_{G\subsetneq F}\succ\rho_G\prec\tau_{G\subsetneq F'}\succ\rho_{F'}.
\]
If, on the other hand, $F\cap F'=\emptyset$, choose rank-one flats $G\subseteq F$ and $G'\subseteq F'$. By (F3), there is exactly one rank-two flat $H$ that contains $G$ and $G'$, so we can construct a sequence
\[
(\rho_F\prec\tau_{G\subsetneq F}\succ)\rho_G\prec\tau_{G\subsetneq H}\succ\rho_{H}\prec\tau_{G'\subsetneq H}\succ\rho_{G'}(\prec\tau_{G'\subsetneq F'}\succ\rho_{F'}),
\]
where the parenthetical pieces should be omitted if $G=F$ or $G'=F'$.

Now consider any star fan $\Sigma_{\sM}^{\sigma_\cF}$ where $\cF=(F_1\subsetneq\cdots\subsetneq F_k)$ with $k\leq d-2$. Notice that such a star fan has dimension at least two, and we can write it as a product of Bergman fans on matroid minors
\[
\Sigma_\sM^{\sigma_\cF}=\prod_{j=0}^k\Sigma_{\sM[F_j,F_{j+1}]}.
\]
Consider two rays $\rho,\rho'\in\Sigma_\sM^{\sigma_\cF}(1)$. If the two rays happen to come from different factors in the product, then we can connect them through the sequence
\[
\rho\prec\rho\times\rho'\succ\rho'.
\]
If, on the other hand, they lie in the same factor, there are two cases to consider. If the matroid minor of the factor that the rays lie in has rank at least $3$, then the rays can be connected via the argument above. If, on the other hand, the matroid minor has rank $2$, then one of the other matroid minors must also have rank at least $2$. Choosing any ray $\rho''$ in the Bergman fan of the second matroid minor, we can connect $\rho$ and $\rho'$ through the sequence
\[
\rho\prec\rho\times\rho''\succ\rho''\prec\rho'\times\rho''\succ\rho'.\qedhere
\]
\end{proof}

In order to verify Condition (ii) of Theorem~\ref{thm:reduce}, there are two cases to consider, depending on whether the two-dimensional star fan in question is, itself, a Bergman fan, or whether it is the product of two one-dimensional Bergman fans. In both cases, we use the fact that, in order to prove that the Hessian of a quadratic form $f\in\R[x_1,\dots,x_n]$ has exactly one eigenvalue, it suffices (by Sylvester's Law of Inertia) to find an invertible change of variables $y_1(x),\dots,y_n(x)$ such that
\[
f=\sum_{i=1}^na_iy_i(x)^2
\]
with exactly one positive $a_i$. We now consider the two cases in the following two lemmas.

\begin{lemma}\label{lem:cond2a}
If $\sM$ is a rank-three matroid, then the Hessian of $\deg_{\Sigma_\sM}(D(z)^2)$ has exactly one positive eigenvalue.
\end{lemma}

\begin{proof}
For a flat $F\in\cL^*$, we use the shorthand $X_{F}=X_{\rho_F}$ and $z_F=z_{\rho_F}$. In order to compute $\deg_{\Sigma_\sM}(D(z)^2)$, we must compute $\deg_{\Sigma_\sM}(X_FX_G)$ for any two flats $F,G\in\cL^*$. If $F\subsetneq G$, then the degree is one, by definition of the degree function, and if $F$ and $G$ are incomparable, then the degree is zero. Thus, it remains to compute the degree of the squared terms. Using the definition of $A^\bullet(\Sigma_{\sM})$ and the flat axioms, the reader is encouraged to verify that  $\deg_{\Sigma_\sM}(X_F^2)=1-|\{G\in\cL^*\mid F\subsetneq G\}|$ if $\rk(F)=1$ and $\deg_{\Sigma_\sM}(X_G^2)=-1$ if $\rk(G)=2$. It follows that
\[
\deg_{\Sigma_\sM}(D(z)^2)=2\sum_{F,G\in\cL^*\atop F\subsetneq G} z_Fz_G+\sum_{F\in\cL*\atop \rk(F)=1}z_F^2-\sum_{F,G\in\cL^*\atop F\subsetneq G}z_F^2-\sum_{G\in\cL*\atop \rk(G)=2}z_G^2.
\]
By creatively organizing the terms, we can rewrite this as
\[
\deg_{\Sigma_\sM}(D(z)^2)=\Big(\sum_{F\in\cL^*\atop\rk(F)=1}z_F\Big)^2-\sum_{G\in\cL^*\atop \rk(G)=2}\Big(z_G-\sum_{F\in\cL^*\atop F\subsetneq G}z_F\Big)^2,
\]
where the only key matroid assertion used in the equivalence of these two formulas is that there exists a unique rank-two flat containing any two distinct rank-one flats. Sylvester's Law of Inertia implies that the Hessian of this quadratic form has exactly one positive eigenvalue.
\end{proof}

\begin{lemma}\label{lem:cond2b}
If $\sM$ and $\sM'$ are rank-two matroids, then the Hessian of $\deg_{\Sigma_{\sM}\times\Sigma_{\sM'}}(D(z)^2)$ has exactly one positive eigenvalue.
\end{lemma}

\begin{proof}
By definition of $A^\bullet(\Sigma_{\sM}\times\Sigma_{\sM'})$, the reader is encouraged to verify that
\[
\deg_{\Sigma_{\sM}\times\Sigma_{\sM'}}(X_\rho X_\eta)=\begin{cases}
0 & \rho,\eta\in\Sigma_{\sM}(1)\text{ or }\rho,\eta\in\Sigma_{\sM'}(1),\\
1 & \rho\in\Sigma_{\sM}(1)\text{ and }\eta\in\Sigma_{\sM'}(1).
\end{cases}
\]
Therefore,
\[
\deg_{\Sigma_{\sM}\times\Sigma_{\sM'}}(D(z)^2)=\sum_{\rho\in\Sigma_{\sM}(1),\;\eta\in\Sigma_{\sM'}(1)}2z_\rho z_\eta,
\]
which can be rewritten as
\[
\deg_{\Sigma_{\sM}\times\Sigma_{\sM'}}(D(z)^2)=\frac{1}{2}\Big(\sum_{\rho\in\Sigma_{\sM}(1)}z_\rho+\sum_{\eta\in\Sigma_{\sM'}(1)}z_\rho\Big)^2-\frac{1}{2}\Big(\sum_{\rho\in\Sigma_{\sM}(1)}z_\rho-\sum_{\eta\in\Sigma_{\sM'}(1)}z_\rho\Big)^2.
\]
Sylvester's Law of Inertia implies that the Hessian of this quadratic form has exactly one positive eigenvalue.
\end{proof}

We now have all the ingredients we need to prove Theorem~\ref{thm:bergman}.

\begin{proof}[Proof of Theorem~\ref{thm:bergman}]
We prove that Bergman fans satisfy the two conditions of Theorem~\ref{thm:reduce}. That Bergman fans satisfy Condition (i) is the content of Lemma~\ref{lem:cond1}. To prove Condition (ii), we first note that, since Bergman fans are balanced, their star fans are also balanced, so Theorem~\ref{thm:vol=deg} implies that the volume polynomials in Condition (ii) are independent of $*$ and are equal to
\[
\deg_{\Sigma_\sM^{\sigma_\cF}}(D(z)^2).
\]
By the product decomposition of star fans given in \eqref{eq:product}, $\Sigma_\sM^{\sigma_\cF}$ is either a two-dimensional Bergman fan or a product of two one-dimensional Bergman fans; in the former case, the Hessian of the volume polynomial has exactly one positive eigenvalue by Lemma~\ref{lem:cond2a}, and in the latter case, by Lemma~\ref{lem:cond2b}.
\end{proof}

\subsection{Revisiting the Heron--Rota--Welsh Conjecture}

The \textbf{characteristic polynomial} of a matroid $\sM=(E,\cL)$ can be defined by
\[
\chi_\sM(\lambda)=\sum_{S\subseteq E}(-1)^{|S|}\lambda^{\rk(\sM)-\rk(S)}.
\]
It can be checked that $\chi_\sM(\lambda)$ has a root at $\lambda=1$ for any positive-rank matroid, and the \textbf{reduced characteristic polynomial} is defined by
\[
\overline\chi_\sM(\lambda)=\frac{\chi_\sM(\lambda)}{\lambda-1}.
\]
We use the notation $\mu^a(\sM)$ and $\overline\mu^a(\sM)$ for the (unsigned) coefficients of these polynomials:
\[
\chi_\sM(\lambda)=\sum_{a=0}^{\rk(\sM)}(-1)^a\mu^a(\sM)\lambda^{\rk(\sM)-a}\;\;\;\text{ and }\;\;\;\overline\chi_\sM(\lambda)=\sum_{a=0}^{\rk(\sM)-1}(-1)^a\overline\mu^a(\sM)\lambda^{\rk(\sM)-1-a}.
\] 
The Heron--Rota--Welsh Conjecture, developed in \cite{Rota,Heron,Welsh}, asserts that the sequence of nonnegative integers $\mu^0(\sM),\dots,\mu^{\rk(\sM)}(\sM)$ is unimodal and log-concave:
\[
0\leq \mu^0(\sM)\leq \dots\leq \mu^k(\sM)\geq\cdots\geq\mu^{\rk(\sM)}(\sM)\geq 0\;\;\;\text{ for some }\;\;\;k\in\{0,\dots,\rk(\sM)\}
\]
and
\[
\mu^k(\sM)^2\geq\mu^{k-1}(\sM)\mu^{k+1}(\sM)\;\;\;\text{ for every }\;\;\;k\in\{1,\dots,\rk(\sM)-1\}.
\]
The Heron--Rota--Welsh Conjecture was first proved by Adiprasito, Huh, and Katz \cite{AHK}. Our aim here is to show how this result also follows from the developments in this paper.

It is elementary to check that the unimodality and log-concavity of the coefficients of the characteristic polynomial is implied by the analogous properties for the coefficients of the reduced characteristic polynomial. The bridge from characteristic polynomials to the content of this paper, then, is a result of Huh and Katz \cite[Proposition~5.2]{HuhKatz} (see also \cite[Proposition~9.5]{AHK} and \cite[Proposition~3.11]{DastidarRoss}), which asserts that
\[
\overline\mu^a(\sM)=\deg_{\Sigma_\sM}(\alpha^{d-a}\beta^a)
\]
where $\rk(\sM)=d+1$ and $\alpha,\beta\in A^1(\Sigma_\sM)$ are defined by
\[
\alpha=\sum_{e_0\in F}X_F\;\;\;\text{ and }\;\;\;\beta=\sum_{e_0\notin F}X_F
\]
for some $e_0\in E$ (these Chow classes are independent of the choice of $e_0$).

Choose any $e_0\in E$, and let $*\in\Inn(N_\R)$ be the inner product with orthonormal basis $\{u_e\mid e\neq e_0\}\subseteq N_\R=\R^E/\R u_E$. For two flats $F_1,F_2\in\cL^*$, we compute
\[
u_{F_1}*u_{F_2}=\begin{cases}
|F_1\cap F_2| & e_0\notin F_1\text{ and } e_0\notin F_2,\\
-|F_1\cap F_2^c| & e_0\notin F_1\text{ and } e_0\in F_2,\\
|F_1^c\cap F_2^c| & e_0\in F_1\text{ and } e_0\in F_2.\\
\end{cases}
\]
Define $z^\alpha,z^\beta\in\R^{\Sigma_\sM(1)}=\R^{\cL^*}$ by
\[
z^\alpha_F=\begin{cases}
1 & e_0\in F,\\
0 & e_0\notin F,
\end{cases}
\;\;\;\text{ and }\;\;\;
z^\beta_F=\begin{cases}
1 & e_0\notin F,\\
0 & e_0\in F,
\end{cases}
\]
so that $D(z^\alpha)=\alpha$ and $D(z^\beta)=\beta$ in $A^1(\Sigma_\sM)$. The following lemma allows us to connect characteristic polynomials to mixed volumes of normal complexes.

\begin{lemma}
$z^\alpha,z^\beta\in\oCub(\Sigma_\sM,*)$.
\end{lemma}

\begin{proof}
We must argue that $w_{\sigma,*}(z^\alpha),w_{\sigma,*}(z^\beta)\in\sigma$ for every cone $\sigma\in\Sigma_\sM$. Consider a flag $\cF=(F_1\subsetneq\dots\subsetneq F_k)$ corresponding to a cone $\sigma_\cF\in\Sigma_{\sM}$. It suffices to prove that
\begin{equation}\label{wvector1}
w_{\sigma_\cF,*}(z^\alpha)=
\begin{cases}
\frac{1}{|F_k^c|}u_{F_k} & e_0\in F_k\\
0 & e_0\notin F_k,
\end{cases}
\end{equation}
and
\begin{equation}\label{wvector2}
w_{\sigma_\cF,*}(z^\beta)=
\begin{cases}
\frac{1}{|F_1|}u_{F_1} & e_0\notin F_1\\
0 & e_0\in F_1,
\end{cases}
\end{equation}
We verify \eqref{wvector1}; the verification \eqref{wvector2} is similar. 

To verify \eqref{wvector1}, first suppose that $e_0\in F_k$. Then for any $j=1,\dots,k$, it follows from the definition of $*$ that
\[
u_{F_k}*u_{F_j}=
\begin{cases}
|F_k^c| & e_0\in F_j,\\
0 & e_0\notin F_j.
\end{cases}
\]
Using this, we verify that $\frac{1}{|F_k^c|}u_{F_k}$ satisfies the defining equations of $w_{\sigma_\cF,*}(z^\alpha)$:
\[
\frac{1}{|F_k^c|}u_{F_k}*u_{F_j}=z_{F_j}^\alpha\;\;\;\text{ for all }\;\;\;j=1,\dots,k.
\]
Now suppose that $e_0\notin F_k$. Then $e_0\notin F_j$ for any $j=1,\dots,k$, so $z_{F_j}^\alpha=0$. Thus, the defining equation for $w_{\sigma_\cF,*}(z^\alpha)$ become
\[
w_{\sigma_\cF,*}(z^\alpha)*u_{F_j}=0\;\;\;\text{ for all }\;\;\;j=1,\dots,k,
\]
showing that $w_{\sigma_\cF,*}(z^\alpha)=0$.
\end{proof}

It follows from Theorem~\ref{thm:mvol=mdeg} that the coefficients of the reduced characteristic polynomial have a volume-theoretic interpretation:
\[
\overline\mu^a(\sM)=\MVol_{\Sigma_\sM,*}(\underbrace{z^\alpha,\dots,z^\alpha}_{d-a},\underbrace{z^\beta,\dots,z^\beta}_{a}).
\]
By \cite[Proposition~7.4]{NR}, we know that $\Cub(\Sigma_\sM,*)\neq\emptyset$, and since the cubical cone is the interior of the pseudocubical cone, we may approximate $z^\alpha,z^\beta\in\oCub(\Sigma_\sM,*)$ with $z_t^\alpha,z_t^\beta\in\Cub(\Sigma_\sM,*)$ such that
\[
\lim_{t\rightarrow 0}z_t^\alpha=z^\alpha\;\;\;\text{ and }\;\;\;\lim_{t\rightarrow 0}z_t^\beta=z^\beta.
\]
Define
\[
\overline\mu_t^a(\sM)=\MVol_{\Sigma_\sM,*}(\underbrace{z_t^\alpha,\dots,z_t^\alpha}_{d-a},\underbrace{z_t^\beta,\dots,z_t^\beta}_{a}).
\]
By Theorem~\ref{thm:bergman}, we know that $(\Sigma_\sM,*)$ is AF, and the AF inequalities applied to the mixed volumes $\overline\mu_t^a(\sM)$ imply that the sequence $\overline\mu_t^0(\sM),\dots,\overline\mu_t^d(\sM)$ is log-concave. Since mixed volumes of cubical values are positive (Proposition~\ref{prop:positive}), and since all log-concave sequences of positive values are unimodal, we see that the sequence $\overline\mu_t^0(\sM),\dots,\overline\mu_t^d(\sM)$ is also unimodal. Since both unimodality and log-concavity are preserved under limits, we conclude that
\[
\overline\mu^0(\sM),\dots,\overline\mu^d(\sM)
\]
is unimodal and log-concave, verifying the Heron--Rota--Welsh Conjecture.

\bibliographystyle{alpha}
\newcommand{\etalchar}[1]{$^{#1}$}

\end{document}